\documentclass[10pt,a4paper]{article}

\usepackage{amssymb,amsmath,mathrsfs,amsthm}
\usepackage{verbatim}
\usepackage[a4paper,includeheadfoot,left=30mm,right=30mm,top=25mm,bottom=30mm]{geometry} 
\usepackage{xcolor,graphicx}

\newtheorem{thm}{Theorem}[section]
\newtheorem{lem}[thm]{Lemma}
\newtheorem{prop}[thm]{Proposition}

\newtheorem{defn}[thm]{Definition}

\theoremstyle{definition}

\newtheorem{rmk}{Remark}
\theoremstyle{remark}

\newcommand{\hsforall}{\hspace{1mm}\forall\hspace{1mm}}						 
\newcommand{\hsexists}{\hspace{1mm}\exists\hspace{1mm}} 					 
\newcommand{\res}{\operatorname*{Res}}                             
\renewcommand{\Re}{\operatorname*{Re}}                             
\renewcommand{\Im}{\operatorname*{Im}}                             
\renewcommand{\d}{\ensuremath{\,\mathrm{d}}}							         
\newcommand{\Mspacer}{\hspace{0.5mm}}                              
\newcommand{\M}[3]{#1_{#2\Mspacer#3}}                              
\newcommand{\Msup}[4]{#1_{#2\Mspacer#3}^{#4}}                      
\newcommand{\Msups}[5]{#1_{#2\Mspacer#3}^{#4\Mspacer#5}}           
\renewcommand{\geq}{\geqslant}                                     
\renewcommand{\leq}{\leqslant}                                     
\renewcommand{\epsilon}{\varepsilon}                               
\newcommand{\BE}{\begin{equation}}                                 
\newcommand{\EE}{\end{equation}}                                   
\newcommand{\BES}{\begin{equation*}}                               
\newcommand{\EES}{\end{equation*}}                                 
\newcommand{\BP}{\begin{pmatrix}}                                  
\newcommand{\EP}{\end{pmatrix}}                                    
\newcommand{\N}{\mathbb{N}}                                        
\newcommand{\R}{\mathbb{R}}                                        
\newcommand{\C}{\mathbb{C}}                                        

\def\clap#1{\hbox to 0pt{\hss#1\hss}}

\numberwithin{equation}{section}

\title{Evolution PDEs and augmented \\ eigenfunctions. Finite interval}
\author{
D. A. Smith$^1$ and A. S. Fokas$^2$
\\
\footnotesize $^1$ \emph{Corresponding author}, 
\footnotesize Department of Mathematics, University of Michigan, \\ \footnotesize Ann Arbor, MI 48109-1043 USA \\
\footnotesize email\textup{: \texttt{daasmith@umich.edu}} \\
\footnotesize $^2$ DAMTP, University of Cambridge, Cambridge, CB\textup{3} \textup{0}WA, UK
}

\begin{document}
\maketitle

\begin{abstract}
The so-called unified method expresses the solution of an initial-boundary value problem (IBVP) for an evolution PDE in the finite interval in terms of an integral in the complex Fourier (spectral) plane. Simple IBVP, which will be referred to as problems of type~I, can be solved via a classical transform pair. For example, the Dirichlet problem of the heat equation can be solved in terms of the transform pair associated with the Fourier sine series. Such transform pairs can be constructed via the spectral analysis of the associated spatial operator. For more complicated IBVP, which will be referred to as problems of type~II, there does \emph{not} exist a classical transform pair and the solution \emph{cannot} be expressed in terms of an infinite series. Here we pose and answer two related questions: first, does there exist a (non-classical) transform pair capable of solving a type~II problem, and second, can this transform pair be constructed via spectral analysis? The answer to both of these questions is positive and this motivates the introduction of a novel class of spectral entities. We call these spectral entities augmented eigenfunctions, to distinguish them from the generalised eigenfunctions described in the sixties by Gel'fand and his co-authors.
\end{abstract}

\subsubsection*{AMS MSC2010}35P10 (primary), 35C15, 35G16, 47A70 (secondary).

\section{Introduction} \label{sec:Intro}

The basis for the classical transform pairs used to solve initial-boundary value problems (IBVP) for linear evolution PDEs is the expansion of the initial datum in terms of appropriate eigenfunctions of the spatial differential operator. The transform pair diagonalises the associated differential operator in the sense of the classical spectral theorem. The main goal of this paper is to show that the unified method~\cite{Fok2007a,FP2001a,Smi2012a} yields an integral representation, like~\eqref{eqn:introIBVP.solution.2}, which in turn gives rise to a transform pair like~\eqref{eqn:introTrans.1.1}, and furthermore the elucidation of the spectral meaning of such new transform pairs leads to new results in spectral theory.

For concreteness, we illustrate the results with a pair of examples before describing the general cases each example typifies. Specifically, we study the following (IBVP) for the linearized Korteweg-de Vries (linearized KdV) equation:
\newline\noindent{\bf Problem 1}
\begin{subequations} \label{eqn:introIBVP.1}
\begin{align} \label{eqn:introIBVP.1:PDE}
q_t(x,t) + q_{xxx}(x,t) &= 0 & (x,t) &\in (0,1)\times(0,T), \\
q(x,0) &= f(x) & x &\in [0,1], \\
q(0,t) = q(1,t) &= 0 & t &\in [0,T], \\
q_x(1,t) &= q_x(0,t)/2 & t &\in [0,T].
\end{align}
\end{subequations}
\noindent{\bf Problem 2}
\begin{subequations} \label{eqn:introIBVP.2}
\begin{align} \label{eqn:introIBVP.2:PDE}
q_t(x,t) + q_{xxx}(x,t) &= 0 & (x,t) &\in (0,1)\times(0,T), \\
q(x,0) &= f(x) & x &\in [0,1], \\
q(0,t) = q(1,t) = q_x(1,t) &= 0 & t &\in [0,T].
\end{align}
\end{subequations}
It is shown in~\cite{FP2001a,Pel2005a,Smi2012a,Smi2013a} that these problems are well-posed and that their solutions can be expressed in the form
\BE \label{eqn:introIBVP.solution.2}
q(x,t) = \frac{1}{2\pi}\left\{\int_{\Gamma^+} + \int_{\Gamma_0}\right\} e^{i\lambda x + i\lambda^3t} \frac{\zeta^+(\lambda)}{\Delta(\lambda)}\d\lambda + \frac{1}{2\pi}\int_{\Gamma^-} e^{i\lambda(x-1) + i\lambda^3t} \frac{\zeta^-(\lambda)}{\Delta(\lambda)}\d\lambda,
\EE
where $\Gamma_0$ and $\Gamma^\pm$ are the contours shown on figure~\ref{fig:3o-cont}, with the regions to their left shaded, are defined by
\begin{align*}
\Gamma_0 &\mbox{ is the circular contour of radius } \frac{1}{2} \mbox{ centred at $0$}, \\
\Gamma^\pm &\mbox{ are the boundaries of the domains } \{\lambda\in\C^\pm:\Im(\lambda^3)>0 \mbox{ and } |\lambda|>1\},
\end{align*}
$\alpha$ is the root of unity $e^{2\pi i/3}$, $\hat{f}(\lambda)$ is the Fourier transform
\BE
\int_0^1 e^{-i\lambda x}f(x)\d x, \qquad \lambda\in\C
\EE
and $\zeta^\pm(\lambda)$, $\Delta(\lambda)$ are defined as follows for all $\lambda\in\C$:
\newline\noindent{\bf Problem 1}
\begin{subequations} \label{eqn:introIBVP.DeltaZeta.1}
\begin{align}
\Delta(\lambda) &= e^{i\lambda} + \alpha e^{i\alpha\lambda} + \alpha^2 e^{i\alpha^2\lambda} + 2(e^{-i\lambda} + \alpha e^{-i\alpha\lambda} + \alpha^2 e^{-i\alpha^2\lambda}),\\ \notag
\zeta^+(\lambda) &= \hat{f}(\lambda)(e^{i\lambda}+2\alpha e^{-i\alpha\lambda}+2\alpha^2 e^{-i\alpha^2\lambda}) + \hat{f}(\alpha\lambda)(\alpha e^{i\alpha\lambda}-2\alpha e^{-i\lambda}) \\ &\hspace{15em} + \hat{f}(\alpha^2\lambda)(\alpha^2 e^{i\alpha^2\lambda}-2\alpha^2e^{-i\lambda}), \\ \notag
\zeta^-(\lambda) &= -\hat{f}(\lambda)(2 + \alpha^2 e^{-i\alpha\lambda} + \alpha e^{-i\alpha^2\lambda}) - \alpha\hat{f}(\alpha\lambda)(2-e^{-i\alpha^2\lambda}) \\ &\hspace{15em} - \alpha^2\hat{f}(\alpha^2\lambda)(2-e^{-i\alpha\lambda}).
\end{align}
\end{subequations}
\noindent{\bf Problem 2}
\begin{subequations} \label{eqn:introIBVP.DeltaZeta.2}
\begin{align}
\Delta(\lambda) &= e^{-i\lambda} + \alpha e^{-i\alpha\lambda} + \alpha^2 e^{-i\alpha^2\lambda}, \\
\zeta^+(\lambda) &= \hat{f}(\lambda)(\alpha e^{-i\alpha\lambda}+\alpha^2 e^{-i\alpha^2\lambda}) - (\alpha\hat{f}(\alpha\lambda) + \alpha^2\hat{f}(\alpha^2\lambda))e^{-i\lambda}, \\
\zeta^-(\lambda) &= -\hat{f}(\lambda) - \alpha\hat{f}(\alpha\lambda) - \alpha^2\hat{f}(\alpha^2\lambda).
\end{align}
\end{subequations}
Note that the radius $\frac{1}{2}$ in the definition of $\Gamma_0$ is not directly related to the factor $\frac{1}{2}$ which appears in the boundary conditions. The radius was simply chosen sufficiently small to ensure that the contour encloses precisely one zero of $\Delta$.

\begin{figure}
\begin{center}
\includegraphics{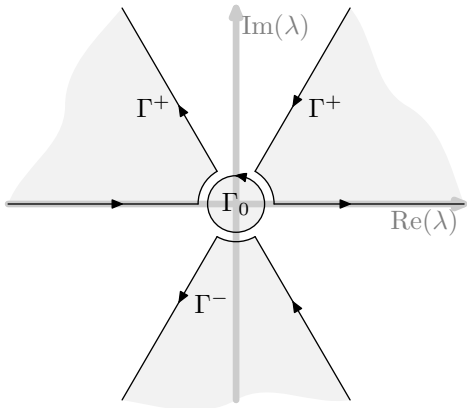}
\caption{Contours for the linearized KdV equation.}
\label{fig:3o-cont}
\end{center}
\end{figure}

For evolution PDEs defined in the finite interval, $x\in[0,1]$, one may expect that the solution can be expressed in terms of an infinite series. However, it is shown in~\cite{Pel2005a,Smi2012a} that for generic boundary conditions this is \emph{impossible}. The solution \emph{can} be expressed in the form of an infinite series only for a particular class of boundary value problems; this class is characterised explicitly in~\cite{Smi2012a}. In particular, problem 2 does \emph{not} belong to this class, in contrast to problem 1 for which there exists the following alternative representation:
\BE \label{eqn:introIBVP.solution.1}
q(x,t) = \frac{1}{2\pi} \sum_{\substack{\sigma\in\overline{\C^+}:\\\Delta(\sigma)=0}} \int_{\Gamma_\sigma} e^{i\lambda x + i\lambda^3t} \frac{\zeta^+(\lambda)}{\Delta(\lambda)}\d\lambda + \frac{1}{2\pi} \sum_{\substack{\sigma\in\C^-:\\\Delta(\sigma)=0}} \int_{\Gamma_\sigma} e^{i\lambda(x-1) + i\lambda^3t} \frac{\zeta^-(\lambda)}{\Delta(\lambda)}\d\lambda,
\EE
where
$$
\Gamma_\sigma \mbox{ is the circular contour centered at } \sigma \mbox{ with radius } \frac{1}{2};
$$
the asymptotic formula for $\sigma$ is given in~\cite{Smi2013a}. By using the residue theorem, it is possible to express the right hand side of equation~\eqref{eqn:introIBVP.solution.1} in terms of an infinite series over $\sigma$.

We note that even for problems for which there does exist a series representation (like problem~1), the integral representation~\eqref{eqn:introIBVP.solution.2} has certain advantages. In particular, via an appropriate contour deformation, one may ensure that the integrand decays exponentially. This yields a rapidly converging solution representation suitable for an efficient numerical evaluation. In the case of initial data with explicitly computable Fourier transform, this is demonstrated in~\cite{FF2008a}, but the rapid convergence is not restricted to such data.

Generic IBVP for which there does \emph{not} exist an infinite series representation will be referred to as problems of type~II, in contrast to those problems whose solutions possess both an integral and a series representation, which will be referred to as problems of type~I,
\begin{quote}
existence of a series representation: type~I \\
existence of only an integral representation: type~II.
\end{quote}

\subsubsection*{Transform pair}

Simple IBVP for linear evolution PDEs can be solved via an appropriate transform pair. For example, the Dirichlet and Neumann problems of the heat equation on the finite interval can be solved with the transform pair associated with the Fourier-sine and the Fourier-cosine series, respectively. Similarly, the series that can be constructed using the residue calculations of the right hand side of equation~\eqref{eqn:introIBVP.solution.1} can be obtained directly via a classical transform pair, which in turn can be constructed via standard spectral analysis.

It turns out that the unified method provides an algorithmic way for constructing a transform pair tailored for a given IBVP. For example, the integral representation~\eqref{eqn:introIBVP.solution.2} gives rise to the following transform pair tailored for solving problems~1 and~2:
\begin{subequations} \label{eqn:introTrans.1.1}
\begin{align} \label{eqn:introTrans.1.1a}
f(x) &\mapsto F(\lambda): & F_\lambda(f) &= \begin{cases} \int_0^1 \phi^+(x,\lambda)f(x)\d x & \mbox{if } \lambda\in\Gamma^+\cup\Gamma_0, \\ \int_0^1 \phi^-(x,\lambda)f(x)\d x & \mbox{if } \lambda\in\Gamma^-, \end{cases} \\ \label{eqn:introTrans.1.1b}
F(\lambda) &\mapsto f(x): & f_x(F) &= \left\{ \int_{\Gamma_0} + \int_{\Gamma^+} + \int_{\Gamma^-} \right\} e^{i\lambda x} F(\lambda) \d\lambda, \qquad x\in[0,1],
\end{align}
where for problems~1 and~2 respectively, $\phi^\pm$ are given by
\begin{align} \notag
\phi^+(x,\lambda)   &= \frac{1}{2\pi\Delta(\lambda)} \left[ e^{-i\lambda x}(e^{i\lambda}+2\alpha e^{-i\alpha\lambda}+2\alpha^2 e^{-i\alpha^2\lambda}) \right. \\ &\hspace{2em} \left. + e^{-i\alpha\lambda x}(\alpha e^{i\alpha\lambda}-2\alpha e^{-i\lambda}) + e^{-i\alpha^2\lambda x}(\alpha^2 e^{i\alpha^2\lambda}-2\alpha^2e^{-i\lambda}) \right], \label{eqn:introTrans.1.1c} \\ \notag
\phi^-(x,\lambda)   &= \frac{-e^{-i\lambda}}{2\pi\Delta(\lambda)} \left[ e^{-i\lambda x}(2 + \alpha^2 e^{-i\alpha\lambda} + \alpha e^{-i\alpha^2\lambda}) + \alpha e^{-i\alpha\lambda x}(2-e^{-i\alpha^2\lambda}) \right. \\ &\hspace{14em} \left. + \alpha^2 e^{-i\alpha^2\lambda x}(2-e^{-i\alpha\lambda}) \right] \label{eqn:introTrans.1.1d}
\end{align}
and
\begin{align} \label{eqn:introTrans.1.1e}
\phi^+(x,\lambda)   &= \frac{1}{2\pi\Delta(\lambda)} \left[ e^{-i\lambda x}(\alpha e^{-i\alpha\lambda}+\alpha^2 e^{-i\alpha^2\lambda}) - (\alpha e^{-i\alpha\lambda x} + \alpha^2 e^{-i\alpha^2\lambda x})e^{-i\lambda} \right], \\ \label{eqn:introTrans.1.1f}
\phi^-(x,\lambda)   &= \frac{-e^{-i\lambda}}{2\pi\Delta(\lambda)} \left[ e^{-i\lambda x} + \alpha e^{-i\alpha\lambda x} + \alpha^2 e^{-i\alpha^2\lambda x} \right].
\end{align}
\end{subequations}
The alternative representation~\eqref{eqn:introIBVP.solution.1} gives rise to the following alternative transform pair tailored for solving problem~1:
\begin{align} \label{eqn:introTrans.1.2}
F(\lambda) &\mapsto f(x): & f^\Sigma_x(F) = \sum_{\substack{\sigma\in\C:\\\Delta(\sigma)=0}}\int_{\Gamma_\sigma} e^{i\lambda x}F(\lambda)\d\lambda,
\end{align}
where $F_\lambda(f)$ is defined by equations~\eqref{eqn:introTrans.1.1a},~\eqref{eqn:introTrans.1.1c} and~\eqref{eqn:introTrans.1.1d} and $\Gamma_\sigma$ is defined below~\eqref{eqn:introIBVP.solution.1}.

The validity of these transform pairs is established in section~\ref{sec:Transforms.valid}. The solution of problems~1 and~2 is then given by
\BE \label{eqn:introIBVP.solution.transform.1}
q(x,t) = f_x\left(e^{i\lambda^3t}F_\lambda(f)\right).
\EE

\subsubsection*{Spectral representation}

Suppose we seek traditional eigenfunctions of the spatial differential operator $S$ associated with the half-line Dirichlet problem for the heat equation, given by
\BE
(Sf)(x) = -f''(x), \qquad \hsforall f\in \mathcal{S}[0,\infty) \mbox{ such that } f(0) = 0,
\EE
where $\mathcal{S}[0,\infty)$ is the restriction of the Schwartz space of rapidly decaying functions to $[0,\infty)$. Then $-f''(x) = \lambda^2f(x)$ implies $f(x) = Ae^{i\lambda x} + Be^{-i\lambda x}$ and the boundary condition yields $B=-A$ hence
\BE
f(x) = A'\sin(\lambda x).
\EE
But, for $f\in S[0,\infty)$, we must have $A'=0$ so there are no nonzero eigenfunctions of $S$.

Instead, one must consider eigen\emph{functionals} or generalised eigenfunctions, following Gel'fand and coauthors~\cite{GS1967a,GV1964a}. Searching for functions $F\in(\mathcal{S}[0,\infty))'$, we find
\BE
F_\lambda(Sf) = \lambda^2F_\lambda(f)
\EE
holds for each $\lambda\in\R$, where
\BE
F_\lambda(f) = \frac{2}{\pi}\int_0^\infty\sin(\lambda x)f(x)\d x.
\EE
That is, the generalised eigenfunctions are precisely the (forward) transforms used to solve the corresponding IBVP. It is therefore reasonable to ask if the transform pair obtained through the unified transform method for IBVP on the finite interval has a similar spectral meaning.

Unfortunately, the concept of generalised eigenfunctions is inadequate for analysing the IBVP studied here because our problems are in general non-self-adjoint. Although the given formal differential operator is self-adjoint, the boundary conditions are in general not self-adjoint. In what follows, we introduce the notion of~\emph{augmented eigenfunctions}. Actually, in order to analyse type~I and type~II IBVP, we introduce two types of augmented eigenfunctions. Type~I are a slight generalisation of the eigenfunctions described by Gel'fand and Vilenkin and are also related with the notion of pseudospectra~\cite{ET2005a}. However, it appears that type~II eigenfunctions comprise a new class of spectral functionals.

\begin{defn} \label{defn:AugEig}
Let $I$ be an open real interval and let $C$ be a linear topological space of functions defined on the closure of $I$ with sufficient smoothness and decay conditions. Let $\Phi\subseteq C$ and let $L:\Phi\to C$ be a linear operator. Let $\gamma$ be an oriented contour in $\C$ and let $E=\{E_\lambda:\lambda\in\gamma\}$ be a family of functionals $E_\lambda\in C'$. Suppose there exist corresponding \emph{remainder} functionals $R_\lambda\in\Phi'$ and \emph{eigenvalues} $\lambda^n$ such that
\BE \label{eqn:defnAugEig.AugEig}
E_\lambda(L\phi) = \lambda^n E_\lambda(\phi) + R_\lambda(\phi), \qquad\hsforall\phi\in\Phi, \hsforall \lambda\in\gamma.
\EE

If
\BE \label{eqn:defnAugEig.Control1}
\int_\gamma e^{i\lambda x} R_\lambda(\phi)\d\lambda = 0, \qquad \hsforall \phi\in\Phi, \hsforall x\in I,
\EE
then we say $E$ is a family of \emph{type~\textup{I} augmented eigenfunctions} of $L$ up to integration along $\gamma$.

If
\BE \label{eqn:defnAugEig.Control2}
\int_\gamma \frac{e^{i\lambda x}}{\lambda^n}  R_\lambda(\phi)\d\lambda = 0, \qquad \hsforall \phi \in\Phi, \hsforall x\in I, 
\EE
then we say $E$ is a family of \emph{type \textup{II}~augmented eigenfunctions} of $L$ up to integration along $\gamma$.
\end{defn}
We note that the class of families of augmented eigenfunctions of a given operator is closed under union.

In the theory of pseudospectra it is required that the norm of the functional $R_\lambda(\phi)$ is small, whereas in our definition it is required that the integral of $\exp(i\lambda x)R_\lambda(\phi)$ along the contour $\gamma$ vanishes. Recall that the inverse transform of the relevant transform pair is defined in terms of a contour integral, thus the above definition is sufficient for our needs.

It will be shown in Section~\ref{sec:Spectral} that $\{F_\lambda:\lambda\in\Gamma_\sigma\hsexists\sigma\in\C:\Delta(\sigma)=0\}$ is a family of type~I augmented eigenfunctions of the differential operator representing the spatial part of problem~1 with eigenvalue $\lambda^3$. Similarly $\{F_\lambda:\lambda\in\Gamma_0\}$ is a family of type~I augmented eigenfunctions of the spatial operator in problem~\eqref{eqn:introIBVP.2}. However, $\{F_\lambda:\lambda\in\Gamma^+\cup\Gamma^-\}$ is a family of type~II augmented eigenfunctions.

\subsubsection*{Diagonalisation of the operator}

Our definition of augmented eigenfunctions, in contrast to the generalized eigenfunctions of Gel'fand and Vilenkin~\cite[Section 1.4.5]{GV1964a}, allows the occurence of remainder functionals. However, the contribution of these remainder functionals is eliminated by integrating over $\gamma$. Hence, integrating equation~\eqref{eqn:defnAugEig.AugEig} over $\gamma$ gives rise to a non-self-adjoint analogue of the spectral representation of an operator.

\begin{defn}
We say that $E=\{E_\lambda:\lambda\in\gamma\}$ is a \emph{complete} family of functionals $E_\lambda\in C'$ if
\BE \label{eqn:defnGEInt.completecriterion}
\phi\in\Phi \mbox{ and } E_\lambda\phi = 0 \hsforall \lambda\in\gamma \quad \Rightarrow \quad \phi=0.
\EE
\end{defn}

Gel'fand~\cite{GV1964a} showed that for any self-adjoint operator the generalised eigenfunctions form a complete system. It is well-known~\cite{Nai1967a} that the generalised eigenfunctions of a Birkhoff-regular operator form a Reisz basis and Locker~\cite{Loc2000a,Loc2008a} provides a constructive proof that the generalised eigenfunctions of a two-point linear differential operator form a complete system, provided the operator is regular or simply irregular (see Locker~\cite{Loc2008a} for definitions of reguarity). However the picture is less clear for degenerate irregular operators such as that associated with problem~2.

It has been established that the generalised eigenfunctions of $S$ associated with problem~2 form a complete system~\cite{Shk1976a} but it is not clear whether this result holds for a general degenerate irregular $S$ associated with a well-posed IBVP. Moreover, completeness of a system is of little use for the expansion of an arbitrary function in that system unless the expansion is guaranteed to converge, and this is known to fail for the operator associated with problem~2~\cite{Jac1915a}.

In contrast, augmented eigenfunctions have the requisite convergence property (and are guaranteed complete) for any $S$ associated with a well-posed IBVP. A very strong property of this convergence is elucidated by the following definitions: two different types of spectral representation of the non-self-adjoint differential operators we study in this paper.

\begin{defn} \label{defn:Spect.Rep.I}
Suppose that $E=\{E_\lambda:\lambda\in\gamma\}$ is a system of type~\textup{I} augmented eigenfunctions of $L$ up to integration over $\gamma$, and that
\BE \label{eqn:Spect.Rep.defnI.conv}
\int_\gamma e^{i\lambda x} E_\lambda L\phi \d\lambda \mbox{\textnormal{ converges }} \hsforall \phi\in\Phi, \hsforall x\in I.
\EE
Furthermore, assume that $E$ is a complete system. Then we say that $E$ provides a \emph{spectral representation} of $L$ in the sense that
\BE \label{eqn:Spect.Rep.I}
\int_\gamma e^{i\lambda x} E_\lambda L\phi \d\lambda = \int_\gamma e^{i\lambda x} \lambda^n E_\lambda \phi\d\lambda \qquad \hsforall \phi\in\Phi, \hsforall x\in I.
\EE
\end{defn}

\begin{defn} \label{defn:Spect.Rep.I.II}
Suppose that $E^{(\mathrm{I})}=\{E_\lambda:\lambda\in\gamma^{(\mathrm{I})}\}$ is a system of type~\textup{I} augmented eigenfunctions of $L$ up to integration over $\gamma^{(\mathrm{I})}$ and that
\BE \label{eqn:Spect.Rep.defnI.II.conv1}
\int_{\gamma^{(\mathrm{I})}} e^{i\lambda x} E_\lambda L\phi \d\lambda \mbox{\textnormal{ converges }} \hsforall \phi\in\Phi, \hsforall x\in I.
\EE
Suppose also that $E^{(\mathrm{II})}=\{E_\lambda:\lambda\in\gamma^{(\mathrm{II})}\}$ is a system of type~\textup{II} augmented eigenfunctions of $L$ up to integration over $\gamma^{(\mathrm{II})}$ and that
\BE \label{eqn:Spect.Rep.defnI.II.conv2}
\int_{\gamma^{(\mathrm{II})}} e^{i\lambda x} E_\lambda \phi \d\lambda \mbox{\textnormal{ converges }} \hsforall \phi\in\Phi, \hsforall x\in I.
\EE
Furthermore, assume that $E=E^{(\mathrm{I})}\cup E^{(\mathrm{II})}$ is a complete system. Then we say that $E$ provides a \emph{spectral representation} of $L$ in the sense that
\begin{subequations} \label{eqn:Spect.Rep.II}
\begin{align} \label{eqn:Spect.Rep.II.1}
\int_{\gamma^{(\mathrm{I})}} e^{i\lambda x} E_\lambda L\phi \d\lambda &= \int_{\gamma^{(\mathrm{I})}} \lambda^n e^{i\lambda x} E_\lambda \phi \d\lambda & \hsforall \phi &\in\Phi, \hsforall x\in I, \\ \label{eqn:Spect.Rep.II.2}
\int_{\gamma^{(\mathrm{II})}} \frac{1}{\lambda^n} e^{i\lambda x} E_\lambda L\phi \d\lambda &= \int_{\gamma^{(\mathrm{II})}} e^{i\lambda x} E_\lambda \phi \d\lambda & \hsforall \phi &\in\Phi, \hsforall x\in I.
\end{align}
\end{subequations}
\end{defn}

Completeness is an essential component of any definition of a spectral representation; see Gel'fand's definition~\cite{GV1964a}. Indeed, otherwise, for some $\phi\in\Phi$, equation~\eqref{eqn:Spect.Rep.I} is trivially $0=0$. Crucially, it is possible to obtain the necessary completeness result by studying the IBVP associated with the operator $L$.

According to definition~\ref{defn:Spect.Rep.I}, the operator $L$ is diagonalised (in the traditional sense) by the complete transform pair
\BE
\left(E_\lambda,\int_\gamma e^{i\lambda x} \cdot \d\lambda \right).
\EE
Hence, augmented eigenfunctions of type~\textup{I} provide a natural extension of the generalised eigenfunctions of Gel'fand~\&~Vilenkin. This form of spectral representation is sufficient to describe the transform pair associated with problem~1. However, the spectral interpretation of the transform pair used to solve problem~2 gives rise to augmented eigenfunctions of type~II, which are clearly quite different from the generalised eigenfunctions of Gel'fand~\&~Vilenkin.

Definition~\ref{defn:Spect.Rep.I.II} describes how an operator may be decomposed into two parts, one of which is diagonalised in the traditional sense, whereas the other possesses a diagonalised inverse.

\begin{thm} \label{thm:Diag:LKdV:1}
The transform pairs $(F_\lambda,f_x)$ defined in equations~\emph{\eqref{eqn:introTrans.1.1a}--\eqref{eqn:introTrans.1.1d}} and~\eqref{eqn:introTrans.1.1a},~\eqref{eqn:introTrans.1.1b},~\eqref{eqn:introTrans.1.1e} and~\eqref{eqn:introTrans.1.1f} provide spectral representations of the spatial differential operators associated with problems~1 and~2 respectively in the sense of definition~\ref{defn:Spect.Rep.I.II}.
\end{thm}

\begin{thm} \label{thm:Diag:LKdV:2}
The transform pair $(F_\lambda,f^\Sigma_x)$ given by equations~\eqref{eqn:introTrans.1.1a} and~\eqref{eqn:introTrans.1.2}, with kernels~\eqref{eqn:introTrans.1.1c} and~\eqref{eqn:introTrans.1.1d}, provides a spectral representation of the spatial differential operator associated with problem~1 in the sense of definition~\ref{defn:Spect.Rep.I}.
\end{thm}

\begin{rmk} \label{rmk:inhomogeneous.BC}
Both problems~1 and~2 involve homogeneous boundary conditions. It is straightforward to extend the above analysis for problems with inhomogeneous boundary conditions, see the remark at the end of section~\ref{ssec:Transform.Method:LKdV}.
\end{rmk}

\begin{rmk} \label{rmk:Other.Papers}
The results in this paper are all for two-point differential operators and the corresponding finite interval IBVP. However the unified transform method is equally well understood for half-line problems. In~\cite{PS2014a}, half-line results are presented which are complementary to those described herein. The half-line and finite interval results are compared and contrasted in~\cite{Smi2014a}
\end{rmk}

\section{Validity of transform pairs} \label{sec:Transforms.valid}

In section~\ref{ssec:Transforms.valid:LKdV} we will establish the validity of the transform pairs defined by equations~\eqref{eqn:introTrans.1.1}. In section~\ref{ssec:Transforms.valid:LKdV:General} we derive an analogous transform pair for a general IBVP.

\subsection{Linearized KdV} \label{ssec:Transforms.valid:LKdV}

\begin{prop} \label{prop:Transforms.valid:LKdV:2}
Let $F_\lambda(f)$ and $f_x(F)$ be given by equations~\emph{\eqref{eqn:introTrans.1.1a}--\eqref{eqn:introTrans.1.1d}}.
For all $f\in C^\infty[0,1]$ such that $f(0)=f(1)=0$ and $f'(0)=2f'(1)$ and for all $x\in(0,1)$, we have
\BE \label{eqn:Transforms.valid:LKdV:prop2:prob1}
f_x(F_\lambda(f)) = f(x).
\EE
Let $F_\lambda(f)$ and $f_x(F)$ be given by equations~\eqref{eqn:introTrans.1.1a},~\eqref{eqn:introTrans.1.1b},~\eqref{eqn:introTrans.1.1e} and~\eqref{eqn:introTrans.1.1f}.
For all $f\in C^\infty[0,1]$ such that $f(0)=f(1)=f'(1)=0$ and for all $x\in(0,1)$,
\BE \label{eqn:Transforms.valid:LKdV:prop2:prob2}
f_x(F_\lambda(f)) = f(x).
\EE
\end{prop}

\begin{proof}
The definition of the transform pair~\eqref{eqn:introTrans.1.1a}--\eqref{eqn:introTrans.1.1d} implies
\BE \label{eqn:Transforms.valid:LKdV:prop2:proof.1}
f_x(F_\lambda(f)) = \frac{1}{2\pi}\left\{\int_{\Gamma^+} + \int_{\Gamma_0}\right\} e^{i\lambda x} \frac{\zeta^+(\lambda)}{\Delta(\lambda)}\d\lambda + \frac{1}{2\pi}\int_{\Gamma^-} e^{i\lambda(x-1)} \frac{\zeta^-(\lambda)}{\Delta(\lambda)}\d\lambda,
\EE
where $\zeta^\pm$ and $\Delta$ are given by equations~\eqref{eqn:introIBVP.DeltaZeta.1} and the contours $\Gamma^+$, $\Gamma^-$ and $\Gamma_0$ are shown in figure~\ref{fig:3o-cont}.

\begin{figure}
\begin{center}
\includegraphics{LKdV-contours-03}
\caption{Contour deformation for the linearized KdV equation.}
\label{fig:3o-contdef}
\end{center}
\end{figure}

The fastest-growing exponentials in the sectors exterior to $\Gamma^\pm$ are indicated on figure~\ref{fig:3o-contdef}a. Each of these exponentials occurs in $\Delta$ and integration by parts shows that the fastest-growing-terms in $\zeta^\pm$ are the exponentials shown on figure~\ref{fig:3o-contdef}a multiplied by $\lambda^{-2}$. Hence the ratio $\zeta^+(\lambda)/\Delta(\lambda)$ decays for large $\lambda$ within the sector $\pi/3\leq\arg\lambda\leq2\pi/3$ and the ratio $\zeta^-(\lambda)/\Delta(\lambda)$ decays for large $\lambda$ within the sectors $-\pi\leq\arg\lambda\leq-2\pi/3$, $-\pi/3\leq\arg\lambda\leq0$. The relevant integrands are meromorphic functions with poles only at the zeros of $\Delta$. The distribution theory of zeros of exponential polynomials~\cite{Lan1931a} implies that the only poles occur within the sets bounded by $\Gamma^\pm$.

The above observations and Jordan's lemma allow us to deform the relevant contours to the contour $\gamma$ shown on figure~\ref{fig:3o-contdef}b; the red arrows on figure~\ref{fig:3o-contdef}a indicate the deformation direction. Hence equation~\eqref{eqn:Transforms.valid:LKdV:prop2:proof.1} simplifies to
\BE \label{eqn:Transforms.valid:LKdV:prop2:proof.2}
f_x(F_\lambda(f)) = \frac{1}{2\pi}\int_\gamma \frac{e^{i\lambda x}}{\Delta(\lambda)}\left( \zeta^+(\lambda) - e^{-i\lambda}\zeta^-(\lambda) \right) \d\lambda.
\EE
Equations~\eqref{eqn:introIBVP.DeltaZeta.1} imply,
\BE \label{eqn:Transforms.valid:LKdV:prop2:proof.3}
\left( \zeta^+(\lambda) - e^{-i\lambda}\zeta^-(\lambda) \right) = \hat{f}(\lambda)\Delta(\lambda),
\EE
where $\hat{f}$ is the Fourier transform of a piecewise smooth function supported on $[0,1]$. Hence the integrand on the right hand side of equation~\eqref{eqn:Transforms.valid:LKdV:prop2:proof.2} is an entire function, so we can deform the contour $\gamma$ onto the real axis. The usual Fourier inversion theorem completes the proof.

The proof for the transform pair~\eqref{eqn:introTrans.1.1a},~\eqref{eqn:introTrans.1.1b},~\eqref{eqn:introTrans.1.1e} and~\eqref{eqn:introTrans.1.1f} is similar.
\end{proof}

In this particular example, it holds that
\BE
f_0(F_\lambda(\phi)) = f(0), \qquad f_1(F_\lambda(\phi)) = f(1)
\EE
but, for general boundary conditions, this will not always be true. However, the vaues at the endpoints can always be recovered by taking appropriate limits from the interior of the interval.

\subsection{General} \label{ssec:Transforms.valid:LKdV:General}

\subsubsection*{Spatial differential operator}

Let $C=C^\infty[0,1]$ and $B_j:C\to\mathbb{C}$ be the following linearly independent boundary forms
\BE
B_j\phi = \sum_{k=0}^{n-1} \left( \M{b}{j}{k}\phi^{(k)}(0) + \M{\beta}{j}{k}\phi^{(k)}(1) \right), \quad j\in\{1,2,\ldots,n\},
\EE
with boundary coefficients $\M{b}{j}{k}$, $\M{\beta}{j}{k} \in \R$. Let
\BE
\Phi=\{\phi\in C:B_j\phi=0\hsforall j\in\{1,2,\ldots,n\}\}
\EE
and $\{B_j^\star:j\in\{1,2,\ldots,n\}\}$ be a set of adjoint boundary forms with adjoint boundary coefficients $\Msup{b}{j}{k}{\star}$, $\Msup{\beta}{j}{k}{\star} \in \R$. Let $S:\Phi\to C$ be the differential operator defined by
\BE \label{eqn:defn.S}
S\phi(x)=(-i)^n\frac{\d^n\phi}{\d x^n}(x).
\EE
Then $S$ is formally self-adjoint but, in general, does not admit a self-adjoint extension because, in general, $B_j\neq B_j^\star$. Indeed, adopting the notation
\BE
[\phi\psi](x) = (-i)^n\sum_{j=0}^{n-1}(-1)^j(\phi^{(n-1-j)}(x)\overline{\psi}^{(j)}(x)),
\EE
of~\cite[Section~11.1]{CL1955a} and using integration by parts, we find
\BE \label{eqn:S.not.s-a}
((-i\d/\d x)^n\phi,\psi) = [\phi\psi](1) - [\phi\psi](0) + (\phi,(-i\d/\d x)^n\psi), \quad \hsforall \phi,\psi\in C^\infty[0,1].
\EE
If $\phi\in\Phi$, then $\psi$ must satisfy the adjoint boundary conditions in order for $[\phi\psi](1) - [\phi\psi](0) = 0$ to be valid.

\subsubsection*{Initial-boundary value problem}

Associated with $S$ and constant $a\in\C$, we define the following homogeneous IBVP: 
\begin{subequations} \label{eqn:IBVP}
\begin{align} \label{eqn:IBVP.PDE}
(\partial_t + aS)q(x,t) &= 0 & \hsforall (x,t) &\in (0,1)\times(0,T), \\ \label{eqn:IBVP.IC}
q(x,0) &= f(x) & \hsforall x &\in [0,1], \\ \label{eqn:IBVP.BC}
q(\cdot,t) &\in \Phi & \hsforall t &\in [0,T],
\end{align}
\end{subequations}
where $f\in\Phi$ is arbitrary.
Such a problem is ill-posed if (but not only if) the exponential time dependence is unbounded for $\lambda\in\R$, which poses restrictions on $a$. This is equivalent~\cite{FP2001a,Smi2012a} to requiring: if $n$ is odd then $a=\pm i$ and if $n$ is even then $\Re(a)\geq0$.

A full characterisation of well-posedness for all problems~\eqref{eqn:IBVP} is given in~\cite{Pel2004a,Smi2012a,Smi2013a}; for even-order problems, well-posedess depends upon the boundary conditions only, but for odd-order it is often the case that a problem is well-posed for $a=i$ and ill-posed for $a=-i$ or vice versa. Both problems~\eqref{eqn:introIBVP.1} and~\eqref{eqn:introIBVP.2} are well-posed. Note that by well-posed, we mean that there exists a unique solution; we make no claims regarding the continuous dependence of the solution on the data.

In the following definition, we make use of the notion of well- and ill-conditioning. This is unrelated to the concept of the same name in numerical analysis and is formally defined in~\cite{PS2013a}. We say that a problem is well-conditioned if certain ratios decay as their arguments approach $\infty$ from within certain sectors of the complex plane. The ratios are those appearing in the integrands of equation~\eqref{eqn:introIBVP.solution.2}. It is easy to see from that equation that their decay is a necessary condition for well-posedness of the IBVP but, as identified in~\cite{Smi2012a}, conditioning also plays a role in the existence of a series representation for the solution of a well-posed problem.

\begin{defn} \label{defn:Types.of.Problem}
We classify the IBVP~\eqref{eqn:IBVP} into three classes using the definitions of~\cite{PS2013a}:
\begin{description}
  \item[\textnormal{type~I:}]{if the problem for $(S,a)$ is well-posed and the problem for $(S,\bar{a})$ is well-conditioned.}
  \item[\textnormal{type~II:}]{if the problem for $(S,a)$ is well-posed but the problem for $(S,\bar{a})$ is ill-conditioned.}
  \item[\textnormal{ill-posed}]{otherwise.}
\end{description}
We will refer to the operators $S$ associated with problems of type~\textnormal{I} and type~\textnormal{II} (for some $a\in\C$) as operators of \emph{type~I} and \emph{type~II} respectively.
\end{defn}
The spectral theory of type~I operators is well understood in terms of an infinite series representation. Here, we provide an alternative spectral representation of the type~I operators and also provide a suitable spectral representation of the type~II operators.

\subsubsection*{Transform pair}

Let $\alpha = e^{2\pi i/n}$. We define the entries of the matrices $M^\pm(\lambda)$ entrywise by
\begin{subequations} \label{eqn:M.defn}
\begin{align} \label{eqn:M+.defn}
\Msup{M}{k}{j}{+}(\lambda) &= \sum_{r=0}^{n-1} (-i\alpha^{k-1}\lambda)^r \Msup{b}{j}{r}{\star}, \\ \label{eqn:M-.defn}
\Msup{M}{k}{j}{-}(\lambda) &= \sum_{r=0}^{n-1} (-i\alpha^{k-1}\lambda)^r \Msup{\beta}{j}{r}{\star}.
\end{align}
\end{subequations}
Then the matrix $M(\lambda)$, defined by
\BE
\M{M}{k}{j}(\lambda) = \Msup{M}{k}{j}{+}(\lambda) + \Msup{M}{k}{j}{-}(\lambda)e^{-i\alpha^{k-1}\lambda},
\EE
is a realization of Birkhoff's characteristic matrix~\cite{Bir1908b} for the operator adjoint to $S$.

We define $\Delta(\lambda) = \det M(\lambda)$. From the theory of exponential polynomials~\cite{Lan1931a}, we know that the only zeros of $\Delta$ are of finite order and are isolated with positive infimal separation $5\epsilon$, say. We define $\Msups{X}{}{}{l}{j}$ as the $(n-1)\times(n-1)$ submatrix of $M$ with $(1,1)$ entry the $(l+1,j+1)$ entry of $M$.

The transform pair is given by
\begin{subequations} \label{eqn:defn.forward.transform}
\begin{align}
f(x) &\mapsto F(\lambda): & F_\lambda(f) &= \begin{cases} F_\lambda^+(f) & \mbox{if } \lambda\in\Gamma_0^+\cup\Gamma_a^+, \\ F_\lambda^-(f) & \mbox{if } \lambda\in\Gamma_0^-\cup\Gamma_a^-, \end{cases} \\ \label{eqn:defn.inverse.transform.2}
F(\lambda) &\mapsto f(x): & f_x(F) &= \int_{\Gamma} e^{i\lambda x} F(\lambda) \d\lambda, \qquad x\in[0,1],
\end{align}
\end{subequations}
where, for $\lambda\in\C$ such that $\Delta(\lambda)\neq0$,
\begin{subequations} \label{eqn:defn.Fpm.rho}
\begin{align}
F^+_\lambda(f) &= \frac{1}{2\pi\Delta(\lambda)} \sum_{l=1}^n\sum_{j=1}^n (-1)^{(n-1)(l+j)} \det \Msups{X}{}{}{l}{j}(\lambda) \Msup{M}{1}{j}{+}(\lambda) \int_0^1 e^{-i\alpha^{l-1}\lambda x} f(x)\d x, \\
F^-_\lambda(f) &= \frac{-e^{-i\lambda}}{2\pi\Delta(\lambda)} \sum_{l=1}^n\sum_{j=1}^n (-1)^{(n-1)(l+j)} \det \Msups{X}{}{}{l}{j}(\lambda) \Msup{M}{1}{j}{-}(\lambda) \int_0^1 e^{-i\alpha^{l-1}\lambda x} f(x)\d x,
\end{align}
\end{subequations}
and the various contours are defined by
\begin{subequations}
\begin{align}
\Gamma       &= \Gamma_0 \cup \Gamma_a, \\
\Gamma_0     &= \Gamma_0^+ \cup \Gamma_0^-, \\
\Gamma^+_0   &= \bigcup_{\substack{\sigma\in\overline{\C^+}:\\\Delta(\sigma)=0}}C(\sigma,\epsilon), \\
\Gamma^-_0   &= \bigcup_{\substack{\sigma\in\C^-:\\\Delta(\sigma)=0}}C(\sigma,\epsilon), \\
\Gamma_a     &= \Gamma^+_a \cup \Gamma^-_a, \\ \notag
\Gamma^\pm_a\; &\mbox{is the boundary of the domain} \\ &\hspace{5ex}\left\{\lambda\in\C^\pm:\Re(a\lambda^n)>0\right\} \setminus \bigcup_{\substack{\sigma\in\C:\\\Delta(\sigma)=0}}D(\sigma,2\epsilon).
\end{align}
\end{subequations}

\begin{figure}
\begin{center}
\includegraphics{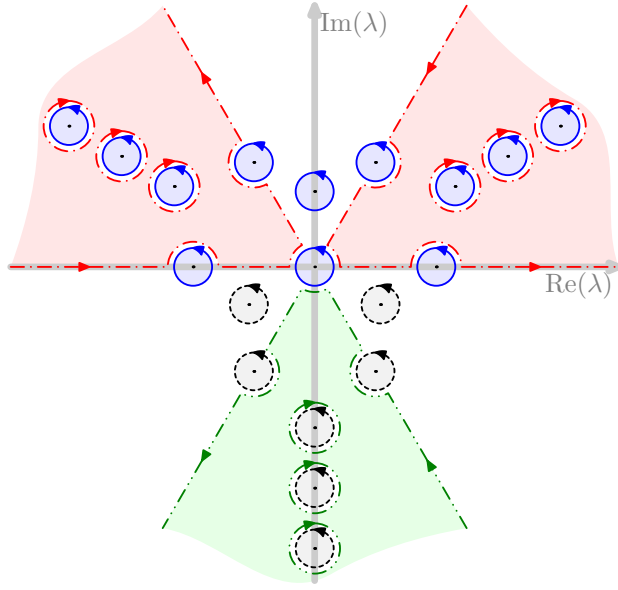}
\caption{Definition of the contour $\Gamma$.}
\label{fig:general-contdef}
\end{center}
\end{figure}

Figure~\ref{fig:general-contdef} shows the position of the contours for some hypothetical $\Delta$ with zeros at the black dots. The contour $\Gamma_0^+$ is shown solid in blue and the contour $\Gamma_0^-$ is shown dashed in black. The contour $\Gamma_a^+$ is shown dot-dashed in red and $\Gamma_a^-$ is shown dot-dot-dashed in green. The regions to the left of each contour are shaded. This case corresponds to $a=-i$. The figure indicates the possibility that there may be infinitely many zeros lying in the interior of the sectors bounded by $\Gamma_a$. For such a zero, $\Gamma_a$ has a circular component enclosing this zero with radius $2\epsilon$.

The validity of the transform pairs is expressed in the following proposition:

\begin{prop} \label{prop:Transforms.valid:General:2}
Let $S$ be a type~\textnormal{I} or type~\textnormal{II} operator. Then for all $f\in\Phi$ and for all $x\in(0,1)$,
\BE \label{eqn:Transforms.valid:General:prop2}
f_x(F_\lambda(f)) = \left\{\int_{\Gamma^+_0}+\int_{\Gamma^+_a}\right\}e^{i\lambda x} F^+_\lambda(f)\d\lambda + \left\{\int_{\Gamma^-_0}+\int_{\Gamma^-_a}\right\}e^{i\lambda x} F^-_\lambda(f)\d\lambda = f(x).
\EE
\end{prop}

\begin{proof}
A simple calculation yields
\BE \label{eqn:F+-F-.is.phihat}
\hsforall f\in C,\hsforall S, \qquad F^+_\lambda(f)-F^-_\lambda(f) = \frac{1}{2\pi}\hat{f}(\lambda).
\EE

As shown in~\cite{Smi2012a}, the well-posedness of the IBVP implies $F^\pm_\lambda(f)=O(\lambda^{-1})$ as $\lambda\to\infty$ within the sectors exterior to $\Gamma^\pm_a$. The only singularities of $F^\pm_\lambda(f)$ are isolated poles hence, by Jordan's lemma and a contour deformation similar to the one shown in figure~\ref{fig:3o-contdef},
\begin{multline} \label{eqn:Transforms.valid:General:2:proof.1}
\left\{\int_{\Gamma^+_0}+\int_{\Gamma^+_a}\right\}e^{i\lambda x} F^+_\lambda(f)\d\lambda + \left\{\int_{\Gamma^-_0}+\int_{\Gamma^-_a}\right\}e^{i\lambda x} F^-_\lambda(f)\d\lambda \\
= \sum_{\substack{\sigma\in\C:\\\Im(\sigma)>\epsilon,\\\Delta(\sigma)=0}}\left\{\int_{C(\sigma,\epsilon)}-\int_{C(\sigma,2\epsilon)}\right\}e^{i\lambda x} F^+_\lambda(f)\d\lambda \hspace{5em} \\ \hspace{5em}+ \sum_{\substack{\sigma\in\C:\\\Im(\sigma)<\epsilon,\\\Delta(\sigma)=0}}\left\{\int_{C(\sigma,\epsilon)}-\int_{C(\sigma,2\epsilon)}\right\}e^{i\lambda x} F^-_\lambda(f)\d\lambda \\
+ \int_\gamma e^{i\lambda x}\left( F^+_\lambda(f) - F^-_\lambda(f) \right) \d\lambda,
\end{multline}
where $\gamma$ is a contour running along the real line in the increasing direction but perturbed along circular arcs in such a way that it is always at least $\epsilon$ away from each pole of $\Delta$. The series on the right hand side of equation~\eqref{eqn:Transforms.valid:General:2:proof.1} yield a zero contribution. As $f\in\Phi$, its Fourier transform $\hat{f}$ is an entire function hence, by statement~\eqref{eqn:F+-F-.is.phihat}, the integrand in the final term on the right hand side of equation~\eqref{eqn:Transforms.valid:General:2:proof.1} is an entire function and we may deform $\gamma$ onto the real line. The validity of the usual Fourier transform completes the proof.
\end{proof}

\section{True integral transform method for IBVP} \label{sec:Transform.Method}

In section~\ref{ssec:Transform.Method:LKdV} we will prove equation~\eqref{eqn:introIBVP.solution.transform.1} for the transform pairs~\eqref{eqn:introTrans.1.1}. In section~\ref{ssec:Transform.Method:LKdV:General}, we establish equivalent results for general type~I and type~II IBVP.

\subsection{Linearized KdV} \label{ssec:Transform.Method:LKdV}

\begin{prop} \label{prop:Transform.Method:LKdV:2}
The solution of problem~1 is given by equation~\eqref{eqn:introIBVP.solution.transform.1}, with $F_\lambda(f)$ and $f_x(F)$ defined by equations~\emph{\eqref{eqn:introTrans.1.1a}--\eqref{eqn:introTrans.1.1d}}.

The solution of problem~2 is given by equation~\eqref{eqn:introIBVP.solution.transform.1}, with $F_\lambda(f)$ and $f_x(F)$ defined by equations~\eqref{eqn:introTrans.1.1a},~\eqref{eqn:introTrans.1.1b},~\eqref{eqn:introTrans.1.1e} and~\eqref{eqn:introTrans.1.1f}.
\end{prop}

\begin{proof}
We present the proof for problem~2. The proof for problem~1 is very similar.

Suppose $q\in C^\infty([0,1]\times[0,T])$ is a solution of the problem~\eqref{eqn:introIBVP.2}. Applying the forward transform to $q$ yields
\BE
F_\lambda(q(\cdot,t)) = \begin{cases} \int_0^1 \phi^+(x,\lambda)q(x,t)\d x & \mbox{if } \lambda\in\overline{\C^+}, \\ \int_0^1 \phi^-(x,\lambda)q(x,t)\d x & \mbox{if } \lambda\in\C^-. \end{cases}
\EE
The PDE and integration by parts imply the following:
\begin{align} \notag
\frac{\d}{\d t} F_\lambda(q(\cdot,t)) &= \int_0^1 \phi^\pm(x,\lambda)q_{xxx}(x,t)\d x \\ \notag
&= - \partial_{x}^2q(1,t) \phi^\pm(1,\lambda) + \partial_{x}^2q(0,t) \phi^\pm(0,\lambda) + \partial_{x}q(1,t) \partial_{x}\phi^\pm(1,\lambda) \\ \notag
&\hspace{5em} - \partial_{x}q(0,t) \partial_{x}\phi^\pm(0,\lambda) - q(1,t) \partial_{x}^2\phi^\pm(1,\lambda) \\
&\hspace{5em} + q(0,t) \partial_{x}^2\phi^\pm(0,\lambda) + i\lambda^3 F_\lambda(q(\cdot,t)).
\end{align}
Rearranging, multiplying by $e^{-i\lambda^3t}$ and integrating, we find
\BE
F_\lambda(q(\cdot,t)) = e^{i\lambda^3t}F_\lambda(f) + e^{i\lambda^3t} \sum_{j=0}^2 (-1)^j\left[\partial_{x}^{2-j}\phi^\pm(0,\lambda) Q_j(0,\lambda) - \partial_{x}^{2-j}\phi^\pm(1,\lambda) Q_j(1,\lambda)\right],
\EE
where
\BE
Q_j(x,\lambda) = \int_0^t e^{-i\lambda^3s} \partial_x^j q(x,s) \d s.
\EE
Evaluating $\partial_{x}^{j}\phi^\pm(0,\lambda)$ and $\partial_{x}^{j}\phi^\pm(1,\lambda)$, we obtain
\begin{multline}
F_\lambda(q(\cdot,t)) = e^{i\lambda^3t}F_\lambda(f) + \frac{e^{i\lambda^3t}}{2\pi} \left[ Q_1(1,\lambda)i\lambda(\alpha-\alpha^2)\frac{e^{i\alpha\lambda}-e^{i\alpha^2\lambda}}{\Delta(\lambda)} \right. \\
\left. + Q_0(0,\lambda)\lambda^2\frac{2e^{-i\lambda}-\alpha e^{-i\alpha\lambda}-\alpha^2 e^{-i\alpha^2\lambda}}{\Delta(\lambda)} \right. \\
\left. + Q_0(1,\lambda)\lambda^2\frac{(\alpha^2-1)e^{i\alpha\lambda}+(\alpha-1)e^{i\alpha^2\lambda}}{\Delta(\lambda)} + Q_2(0,\lambda) + Q_1(0,\lambda) i\lambda \right],
\end{multline}
for all $\lambda\in\Gamma^+\cup\Gamma_0$ and
\begin{multline}
F_\lambda(q(\cdot,t)) = e^{i\lambda^3t}F_\lambda(f) + \frac{e^{-i\lambda+i\lambda^3t}}{2\pi} \left[ Q_1(1,\lambda)i\lambda\frac{e^{-i\lambda}+\alpha^2e^{-i\alpha\lambda}+\alpha e^{-i\alpha^2\lambda}}{\Delta(\lambda)} \right. \\
\left. + Q_0(0,\lambda)\lambda^2\frac{3}{\Delta(\lambda)} - Q_0(1,\lambda)\lambda^2\frac{e^{-i\lambda}+e^{-i\alpha\lambda}+e^{-i\alpha^2\lambda}}{\Delta(\lambda)} + Q_2(1,\lambda) \right],
\end{multline}
for all $\lambda\in\Gamma^-$.

Hence, the validity of the transform pair (proposition~\ref{prop:Transforms.valid:LKdV:2}) implies
\begin{multline} \label{eqn:Transform.Method:LKdV.2:q.big}
q(x,t) = \left\{\int_{\Gamma_0}+\int_{\Gamma^+}+\int_{\Gamma^-}\right\} e^{i\lambda x+i\lambda^3t} F_\lambda(f)\d\lambda \\
+\frac{1}{2\pi}\left\{\int_{\Gamma_0}+\int_{\Gamma^+}\right\} e^{i\lambda x+i\lambda^3t} \left[ Q_1(1,\lambda)i\lambda(\alpha-\alpha^2)\frac{e^{i\alpha\lambda}-e^{i\alpha^2\lambda}}{\Delta(\lambda)} \right. \\
\left. + Q_0(0,\lambda)\lambda^2\frac{2e^{-i\lambda}-\alpha e^{-i\alpha\lambda}-\alpha^2 e^{-i\alpha^2\lambda}}{\Delta(\lambda)} \right. \\
\left. + Q_0(1,\lambda)\lambda^2\frac{(\alpha^2-1)e^{i\alpha\lambda}+(\alpha-1)e^{i\alpha^2\lambda}}{\Delta(\lambda)} \right] \d\lambda \\
+\frac{1}{2\pi}\int_{\Gamma^-} e^{i\lambda (x-1)+i\lambda^3t} \left[ Q_1(1,\lambda)i\lambda\frac{e^{-i\lambda}+\alpha^2e^{-i\alpha\lambda}+\alpha e^{-i\alpha^2\lambda}}{\Delta(\lambda)} \right. \\
\left. + Q_0(0,\lambda)\lambda^2\frac{3}{\Delta(\lambda)} - Q_0(1,\lambda)\lambda^2\frac{e^{-i\lambda}+e^{-i\alpha\lambda}+e^{-i\alpha^2\lambda}}{\Delta(\lambda)} \right] \d\lambda \\
+\frac{1}{2\pi}\left\{\int_{\Gamma_0}+\int_{\Gamma^+}\right\} e^{i\lambda x+i\lambda^3t} \left[ Q_2(0,\lambda) + Q_1(0,\lambda) i\lambda \right] \d\lambda \\
+\frac{1}{2\pi}\int_{\Gamma^-} e^{i\lambda(x-1)+i\lambda^3t} Q_2(1,\lambda) \d\lambda.
\end{multline}
Integration by parts yields
\BE
Q_j(x,\lambda) = O(\lambda^{-3}),
\EE
as $\lambda\to\infty$ within the region enclosed by $\Gamma^\pm$. Hence, by Jordan's lemma, the final two lines of equation~\eqref{eqn:Transform.Method:LKdV.2:q.big} vanish. The boundary conditions imply
\BE
Q_0(0,\lambda) = Q_0(1,\lambda) = Q_1(1,\lambda) = 0,
\EE
so the second, third, fourth, fifth and sixth lines of equation~\eqref{eqn:Transform.Method:LKdV.2:q.big} vanish. Hence
\BE
q(x,t) = \left\{\int_{\Gamma_0}+\int_{\Gamma^+}+\int_{\Gamma^-}\right\} e^{i\lambda x+i\lambda^3t} F_\lambda(f)\d\lambda.
\EE
\end{proof}

The above proof also demonstrates how the transform pair may be used to solve a problem with inhomogeneous boundary conditions: consider the problem,
\begin{subequations} \label{eqn:introIBVP.2inhomo}
\begin{align} \label{eqn:introIBVP.2inhomo:PDE}
q_t(x,t) + q_{xxx}(x,t) &= 0 & (x,t) &\in (0,1)\times(0,T), \\
q(x,0) &= \phi(x) & x &\in [0,1], \\
q(0,t) &= h_1(t) & t &\in [0,T], \\
q(1,t) &= h_2(t) & t &\in [0,T], \\
q_x(1,t) &= h_3(t) & t &\in [0,T],
\end{align}
\end{subequations}
for some given boundary data $h_j\in C^\infty[0,1]$. Then $Q_0(0,\lambda)$, $Q_0(1,\lambda)$ and $Q_1(1,\lambda)$ are nonzero, but they are known quantities, namely $t$-transforms of the boundary data. Substituting these values into equation~\eqref{eqn:Transform.Method:LKdV.2:q.big} yields an explicit expression for the solution.

\subsection{General} \label{ssec:Transform.Method:LKdV:General}

\begin{prop} \label{prop:Transform.Method:General:2}
The solution of a type~\textnormal{I} or type~\textnormal{II} IBVP is given by
\BE \label{eqn:Transform.Method:General:prop2:1}
q(x,t) = f_x \left( e^{-a\lambda^nt} F_\lambda(f) \right).
\EE
\end{prop}

\begin{lem} \label{lem:GEInt1}
Let $f\in\Phi$ and $S$ be the spatial differential operator defined in equation~\eqref{eqn:defn.S}. Then there exist polynomials $P^\pm_f$ of degree at most $n-1$ such that
\begin{subequations}
\begin{align}
F^+_\lambda(Sf) &= \lambda^n F^+_\lambda(f) + P^+_f(\lambda), \\
F^-_\lambda(Sf) &= \lambda^n F^-_\lambda(f) + P^-_f(\lambda) e^{-i\lambda}.
\end{align}
\end{subequations}
\end{lem}

\begin{proof}
Let $(\phi,\psi)$ be the usual inner product $\int_0^1\phi(x)\overline{\psi}(x)\d x$. For any $\lambda\in\Gamma$, we can represent $F^\pm_\lambda$ as the inner product $F^\pm_\lambda(f)=(f,\phi^\pm_\lambda)$, for the function $\phi^\pm_\lambda(x)$, smooth in $x$ and meromorphic in $\lambda$, defined by
\begin{subequations} \label{eqn:defn.fpm.rho}
\begin{align}
\overline{\phi^+_\lambda}(x) &= \frac{1}{2\pi\Delta(\lambda)} \sum_{l=1}^n\sum_{j=1}^n (-1)^{(n-1)(l+j)} \det \Msups{X}{}{}{l}{j}(\lambda) \Msup{M}{1}{j}{+}(\lambda) e^{-i\alpha^{l-1}\lambda x}, \\
\overline{\phi^-_\lambda}(x) &= \frac{-e^{-i\lambda}}{2\pi\Delta(\lambda)} \sum_{l=1}^n\sum_{j=1}^n (-1)^{(n-1)(l+j)} \det \Msups{X}{}{}{l}{j}(\lambda) \Msup{M}{1}{j}{-}(\lambda) e^{-i\alpha^{l-1}\lambda x}.
\end{align}
\end{subequations}

As $\phi^\pm_\lambda$, $Sf\in C^\infty[0,1]$ and $\alpha^{(l-1)n}=1$, equation~\eqref{eqn:S.not.s-a} yields
\BE
F^\pm_\lambda(Sf) = \lambda^n F^\pm_\lambda(f) + [f \phi^\pm_\lambda](1) - [f \phi^\pm_\lambda](0).
\EE
If $B$, $B^\star:C^\infty[0,1]\to\C^n$, are the real vector boundary forms
\BE
B=(B_1,B_2,\ldots,B_n), \qquad B^\star=(B_1^\star,B_2^\star,\ldots,B_n^\star),
\EE
then the boundary form formula~\cite[theorem~11.2.1]{CL1955a} guarantees the existance of complimentary vector boundary forms $B_c$, $B_c^\star$ such that
\BE \label{eqn:propGEInt.3}
[f \phi^\pm_\lambda](1) - [f \phi^\pm_\lambda](0) = Bf \cdot B_c^\star \phi^\pm_\lambda + B_cf \cdot B^\star \phi^\pm_\lambda,
\EE
where $\cdot$ is the sesquilinear dot product. We consider the right hand side of equation~\eqref{eqn:propGEInt.3} as a function of $\lambda$. As $Bf=0$, this expression is a linear combination of the functions $B_k^\star\overline{\phi^\pm}_\lambda$ of $\lambda$, with coefficients given by the complementary boundary forms.

The definitions of $B_k^\star$ and $\phi_\lambda^+$ imply
\begin{align} \notag
B_k^\star\overline{\phi^+_\lambda} &= \frac{1}{2\pi\Delta(\lambda)} \sum_{l=1}^n\sum_{j=1}^n (-1)^{(n-1)(l+j)} \det \Msups{X}{}{}{l}{j}(\lambda) \Msup{M}{1}{j}{+}(\lambda) B_k^\star(e^{-i\alpha^{l-1}\lambda \cdot}) \\ \notag
&= \frac{1}{2\pi\Delta(\lambda)} \sum_{l=1}^n\sum_{j=1}^n (-1)^{(n-1)(l+j)} \det \Msups{X}{}{}{l}{j}(\lambda) \Msup{M}{1}{j}{+}(\lambda) \M{M}{l}{k}(\lambda).
\end{align}
But
\BE
\sum_{l=1}^n (-1)^{(n-1)(l+j)} \det\Msups{X}{}{}{l}{j}(\lambda)\M{M}{l}{k}(\lambda) = \Delta(\lambda)\M{\delta}{j}{k},
\EE
so
\begin{subequations}
\begin{align}
B_k^\star \overline{\phi^+_\lambda} &= \frac{1}{2\pi}\Msup{M}{1}{k}{+}(\lambda). \\
\intertext{Similarly,}
B_k^\star \overline{\phi^-_\lambda} &= \frac{-e^{-i\lambda}}{2\pi}\Msup{M}{1}{k}{-}(\lambda).
\end{align}
\end{subequations}
Finally, by equations~\eqref{eqn:M.defn}, $\Msup{M}{1}{k}{\pm}$ are polynomials of order at most $n-1$.
\end{proof}

\begin{proof}[Proof of proposition~\ref{prop:Transform.Method:General:2}]
Let $q$ be the solution of the problem. Then, since $q$ satisfies the partial differential equation~\eqref{eqn:IBVP.PDE},
\BE
\frac{\d}{\d t} F^+_\lambda(q(\cdot,t)) = -aF^+_\lambda(S(q(\cdot,t))) = -a\lambda^n F^+_\lambda(q(\cdot,t)) - a P^+_{q(\cdot,t)}(\lambda),
\EE
where, by lemma~\ref{lem:GEInt1}, $P^+_{q(\cdot,t)}$ is a polynomial of degree at most $n-1$. Hence
\BE
\frac{\d}{\d t} \left( e^{a\lambda^nt}F^+_\lambda(q(\cdot,t)) \right) = -ae^{a\lambda^nt}P^+_{q(\cdot,t)}(\lambda).
\EE
Integrating with respect to $t$ and applying the initial condition~\eqref{eqn:IBVP.IC}, we find
\BE \label{eqn:Transform.Method:General:1proof.1}
F^+_\lambda(q(\cdot,t)) = e^{-a\lambda^nt}F^+_\lambda(f) -a e^{-a\lambda^nt} \int_0^t e^{a\lambda^ns} P^+_{q(\cdot,s)}(\lambda)\d s.
\EE
Similarly,
\BE \label{eqn:Transform.Method:General:1proof.2}
F^-_\lambda(q(\cdot,t)) = e^{-a\lambda^nt}F^-_\lambda(f) -a e^{-i\lambda-a\lambda^nt} \int_0^t e^{a\lambda^ns} P^-_{q(\cdot,s)}(\lambda)\d s,
\EE
where $P^-_{q(\cdot,t)}$ is another polynomial of degree at most $n-1$. The validity of the type~II transform pair, proposition~\ref{prop:Transforms.valid:General:2}, implies
\begin{multline} \label{eqn:Transform.Method:General:2proof.3}
q(x,t) = \int_{\Gamma^+} e^{i\lambda x-a\lambda^nt}F^+_\lambda(f) \d\lambda + \int_{\Gamma^-} e^{i\lambda(x-1)-a\lambda^nt}F^-_\lambda(f) \d\lambda \\
-a \int_{\Gamma^+_0} e^{i\lambda x-a\lambda^nt} \left(\int_0^t e^{a\lambda^ns} P^+_{q(\cdot,s)}(\lambda) \d s\right) \d\lambda \\
-a \int_{\Gamma^-_0} e^{i\lambda(x-1)-a\lambda^nt} \left(\int_0^t e^{a\lambda^ns} P^-_{q(\cdot,s)}(\lambda) \d s\right) \d\lambda \\
-a \int_{\Gamma^+_a} e^{i\lambda x-a\lambda^nt} \left(\int_0^t e^{a\lambda^ns} P^+_{q(\cdot,s)}(\lambda) \d s\right) \d\lambda \\
-a \int_{\Gamma^-_a} e^{i\lambda(x-1)-a\lambda^nt} \left(\int_0^t e^{a\lambda^ns} P^-_{q(\cdot,s)}(\lambda) \d s\right) \d\lambda.
\end{multline}
As $P^\pm_{q(\cdot,s)}$ are polynomials, the integrands
\BES
e^{i\lambda x-a\lambda^nt} \left(\int_0^t e^{a\lambda^ns} P^+_{q(\cdot,s)}(\lambda) \d s\right) \mbox{ and } e^{i\lambda(x-1)-a\lambda^nt} \left(\int_0^t e^{a\lambda^ns} P^-_{q(\cdot,s)}(\lambda) \d s\right)
\EES
are both entire functions of $\lambda$. Hence the third and fourth terms of equation~\eqref{eqn:Transform.Method:General:2proof.3} vanish. Integration by parts yields
\begin{align*}
e^{i\lambda x-a\lambda^nt} \left(\int_0^t e^{a\lambda^ns} P^+_{q(\cdot,s)}(\lambda) \d s\right) &= O(\lambda^{-1}) \mbox{ as } \lambda\to\infty\; \begin{matrix}\mbox{ within the region }\\\mbox{ enclosed by } \Gamma^+_a,\end{matrix} \\
e^{i\lambda(x-1)-a\lambda^nt} \left(\int_0^t e^{a\lambda^ns} P^-_{q(\cdot,s)}(\lambda) \d s\right) &= O(\lambda^{-1}) \mbox{ as } \lambda\to\infty\; \begin{matrix}\mbox{ within the region }\\\mbox{ enclosed by } \Gamma^-_a.\end{matrix}
\end{align*}
Hence, by Jordan's lemma, the final two terms of equation~\eqref{eqn:Transform.Method:General:2proof.3} vanish.
\end{proof}

\begin{rmk}
The same method may be used to solve IBVP with inhomogeneous boundary conditions. The primary difference is that statement~\eqref{eqn:F+-F-.is.phihat} must be replaced with~\cite[lemma~4.1]{Smi2012a}.
\end{rmk}

\section{Analysis of the transform pair} \label{sec:Spectral}

In this section we analyse the spectral properties of the transform pairs using the notion of augmented eigenfunctions.

\subsection{Linearized KdV} \label{ssec:Spectral:LKdV}

\subsubsection*{Augmented Eigenfunctions}

Let $S^{(\mathrm{I})}$ and $S^{(\mathrm{II})}$ be the differential operators representing the spatial parts of the IBVP~1 and~2, respectively. Each operator is a restriction of the same formal differential operator, $(-i\d/\d x)^3$ to the domain of initial data compatible with the boundary conditions of the problem:
\begin{align} \label{eqn:introS1}
\mathcal{D}(S^{(\mathrm{I})})  &= \{f\in C^\infty[0,1]: f(0)=f(1)=0,\; f'(0)=2f'(1)\}, \\ \label{eqn:introS2}
\mathcal{D}(S^{(\mathrm{II})}) &= \{f\in C^\infty[0,1]: f(0)=f(1)=f'(1)=0\}.
\end{align}

A simple calculation reveals that $\{F_\lambda:\lambda\in\Gamma_0\}$ (where $F_\lambda$ is defined by equations~\eqref{eqn:introTrans.1.1a},~\eqref{eqn:introTrans.1.1c} and~\eqref{eqn:introTrans.1.1d}) is a family of type~I augmented eigenfunctions of $S^{(\mathrm{I})}$. Indeed, integration by parts yields
\BE \label{eqn:introS1.Rphi}
F_\lambda(S^{(\mathrm{I})}f) = \begin{cases} \lambda^3 F_\lambda(f) + \left( - \displaystyle\frac{i}{2\pi}f''(0) + \displaystyle\frac{\lambda}{2\pi}f'(0) \right) & \lambda\in\Gamma^+\cup\Gamma_0, \\ \lambda^3 F_\lambda(f) + \left( - \displaystyle\frac{ie^{-i\lambda}}{2\pi}f''(1) + \displaystyle\frac{\lambda e^{-i\lambda}}{2\pi}f'(1) \right) & \lambda\in\Gamma^-. \end{cases}
\EE
For any $f$, the remainder functional, which is enclosed in parentheses, can be analytically extended onto the regions lying to the left of $\Gamma^+\cup\Gamma^-\cup\Gamma_0$. The contour $\Gamma_0$ is closed and circular hence~\eqref{eqn:defnAugEig.Control1} holds.

In the same way $\{F_\lambda:\lambda\in\Gamma_0\}$, where $F_\lambda$ is defined by equations~~\eqref{eqn:introTrans.1.1a},~\eqref{eqn:introTrans.1.1c} and~\eqref{eqn:introTrans.1.1d} is a family of type~I augmented eigenfunctions of $S^{(\mathrm{II})}$. Indeed
\BE \label{eqn:introS2.Rphi}
F_\lambda(S^{(\mathrm{II})}f) = \begin{cases} \lambda^3 F_\lambda(f) + \left( - \displaystyle\frac{i}{2\pi}f''(0) - \displaystyle\frac{\lambda}{2\pi}f'(0) \right) & \lambda\in\Gamma^+\cup\Gamma_0, \\ \lambda^3 F_\lambda(f) + \left( - \displaystyle\frac{i e^{-i\lambda}}{2\pi}f''(1) \right) & \lambda\in\Gamma^-, \end{cases}
\EE
so the remainder functional is again analytically extensible.

Furthermore, the ratio of the remainder functionals to the eigenvalue is a rational function with no pole in the regions enclosed by $\Gamma^\pm$ and decaying as $\lambda\to\infty$. Jordan's lemma implies~\eqref{eqn:defnAugEig.Control2} hence $\{F_\lambda:\lambda\in\Gamma^+\cup\Gamma^-\}$ is a family of type~II augmented eigenfunctions of the corresponding $S^{(\mathrm{I})}$ or $S^{(\mathrm{II})}$.

\subsubsection*{Spectral representation of $S^{(\mathrm{II})}$}

We have shown above that $\{F_\lambda:\lambda\in\Gamma_0\}$ is a family of type~I augmented eigenfunctions and $\{F_\lambda:\lambda\in\Gamma^+\cup\Gamma^-\}$ is a family of type~II augmented eigenfunctions of $S^{(\mathrm{II})}$, each with eigenvalue $\lambda^3$. It remains to show that the integrals
\BE
\int_{\Gamma_0}e^{i\lambda x}F_\lambda(Sf)\d\lambda, \qquad \int_{\Gamma^+\cup\Gamma^-}e^{i\lambda x}F_\lambda(f)\d\lambda
\EE
converge.

A simple calculation reveals that $F_\lambda(\psi)$ has a removable singularity at $\lambda=0$, for any $\psi\in C$. Hence the first integral not only converges but evaluates to $0$. Thus, the second integral represents $f_x(F_\lambda(f))=f$ and converges by proposition~\ref{prop:Transforms.valid:LKdV:2}.

This completes the proof of theorem~\ref{thm:Diag:LKdV:1} for problem~2.

\subsubsection*{Spectral representation of $S^{(\mathrm{I})}$}

By the above argument, it is clear that the transform pair $(F_\lambda,f_x)$ defined by equations~\eqref{eqn:introIBVP.solution.2} provides a spectral representation of $S^{(\mathrm{I})}$ in the sense of definition~\ref{defn:Spect.Rep.I.II}, verifying theorem~\ref{thm:Diag:LKdV:1} for problem~1.

It is clear that $\{F_\lambda:\lambda\in\Gamma^\pm\}$ is not a family of type~I augmented eigenfunctions, so the representation~\eqref{eqn:introIBVP.solution.2} does \emph{not} provide a spectral representation of $S^{(\mathrm{I})}$ in the sense of definition~\ref{defn:Spect.Rep.I}. However, equation~\eqref{eqn:introIBVP.solution.1} does provide a representation in the sense of definition~\ref{defn:Spect.Rep.I}. Indeed, equation~\eqref{eqn:introIBVP.solution.1} implies that it is possible to deform the contours $\Gamma^\pm$ onto
\BES
\bigcup_{\substack{\sigma\in\C:\\\Delta(\sigma)=0}}\Gamma_\sigma.
\EES
It is possible to make this deformation without any reference to the IBVP. By an argument similar to that in the proof of proposition~2.1, we are able to `close' (whereas in the earlier proof we `opened') the contours $\Gamma^\pm$ onto simple circular contours each enclosing a single zero of $\Delta$. Thus, an equivalent inverse transform is given by~\eqref{eqn:introTrans.1.2}. It is clear that, for each $\sigma$ a zero of $\Delta$, $\{F_\lambda:\lambda\in\Gamma_\sigma\}$ is a family of type~I augmented eigenfunctions of $S^{(\mathrm{I})}$ up to integration over $\Gamma_\sigma$.

It remains to show that the series
\BE \label{eqn:LKdV1.Sphi.Series}
\sum_{\substack{\sigma\in\C:\\\Delta(\sigma)=0}} \int_{\Gamma_\sigma} e^{i\lambda x} F_\lambda(Sf) \d\lambda
\EE
converges. The validity of the transform pair $(F_\lambda,f^\Sigma_x)$ given by~\eqref{eqn:introTrans.1.1a},~\eqref{eqn:introTrans.1.1c},~\eqref{eqn:introTrans.1.1d} and~\eqref{eqn:introTrans.1.2} is insufficient to justify this convergence since, in general, $Sf$ may not satisfy the boundary conditions, so $Sf$ may not be a valid initial datum of the problem. It is possible to circumvent this difficulty, as demonstrated in the proof of theorem~\ref{thm:Diag.2} below, but here we demonstrate convergence directly.

The augmented eigenfunctions $F_\lambda$ are meromorphic functions of $\lambda$, represented in their definition~\eqref{eqn:introTrans.1.1a},~\eqref{eqn:introTrans.1.1c},~\eqref{eqn:introTrans.1.1d} as the ratio of two entire functions, with singularities only at the zeros of the exponential polynomial $\Delta$. The theory of exponential polynomials~\cite{Lan1931a} implies that the only zeros of $\Delta$ are of finite order, so each integral in the series converges and is equal to the residue of the pole at $\sigma$. Furthermore, an asymptotic calculation reveals that these zeros are at $0$, $\alpha^j\lambda_k$, $\alpha^j\mu_k$, for each $j\in\{0,1,2\}$ and $k\in\N$, where
\begin{align} \label{eqn:LKdV.1.lambdak}
\lambda_k &=  \left(2k-\frac{1}{3}\right)\pi + i\log2 + O\left( e^{-\sqrt{3}k\pi} \right), \\
\mu_k     &= -\left(2k-\frac{1}{3}\right)\pi + i\log2 + O\left( e^{-\sqrt{3}k\pi} \right).
\end{align} \label{eqn:LKdV.1.muk}
Evaluating the first derivative of $\Delta$ at these zeros, we find
\begin{align}
\Delta'(\lambda_k) &= (-1)^{k+1}\sqrt{2}e^{ i\frac{\sqrt{3}}{2}\log2} e^{\sqrt{3}\pi(k-1/6)} + O(1), \\
\Delta'(\mu_k)     &= (-1)^{k}  \sqrt{2}e^{-i\frac{\sqrt{3}}{2}\log2} e^{\sqrt{3}\pi(k-1/6)} + O(1).
\end{align}
Hence, at most finitely many zeros of $\Delta$ are of order greater than $1$. A straightforward calculation reveals that $0$ is a removable singularity. Hence, via a residue calculation and integration by parts, we find that we can represent the tail of the series~\eqref{eqn:LKdV1.Sphi.Series} in the form
\begin{subequations}
\begin{multline}
i\sum_{k=N}^\infty \left\{ \frac{1}{\lambda_k\Delta'(\lambda_k)} \left[ e^{i\lambda_kx}\left((Sf)(1)Y_1(\lambda_k)-(Sf)(0)Y_0(\lambda_k)\right) \right.\right. \\
+ \alpha^2 e^{i\alpha\lambda_kx}\left((Sf)(1)Y_1(\alpha\lambda_k)-(Sf)(0)Y_0(\alpha\lambda_k)\right) \\
\hspace{18ex} \left. - \alpha e^{i\alpha^2\lambda_k(x-1)}\left((Sf)(1)Z_1(\alpha^2\lambda_k)-(Sf)(0)Z_0(\alpha^2\lambda_k)\right) \right] \\
+ \frac{1}{\mu_k\Delta'(\mu_k)} \left[ e^{i\mu_kx}\left((Sf)(1)Y_1(\mu_k)-(Sf)(0)Y_0(\mu_k)\right) \right. \hspace{23ex} \\
- \alpha^2 e^{i\alpha\mu_k(x-1)}\left((Sf)(1)Z_1(\alpha\mu_k)-(Sf)(0)Z_0(\alpha\mu_k)\right) \\
\left.\left. + \alpha e^{i\alpha^2\mu_kx}\left((Sf)(1)Y_1(\alpha^2\mu_k)-(Sf)(0)Y_0(\alpha^2\mu_k)\right) \right] + O(k^{-2}) \right\},
\end{multline}
where
\begin{align}
Y_1(\lambda) &= 3 + 2(\alpha^2-1)e^{i\alpha\lambda} + 2(\alpha-1)e^{i\alpha^2\lambda}, \\
Y_0(\lambda) &= e^{i\lambda}+e^{i\alpha\lambda}+e^{i\alpha^2\lambda}-4e^{-i\lambda}+2e^{-i\alpha\lambda}+2e^{-i\alpha^2\lambda}, \\
Z_1(\lambda) &= \alpha e^{i\alpha\lambda} + 2e^{-i\alpha\lambda} + \alpha^2 e^{i\alpha^2\lambda} + 2e^{-i\alpha^2\lambda}, \\
Z_0(\lambda) &= 6 + (\alpha^2-1)e^{-i\alpha\lambda} + (\alpha-1)e^{-i\alpha^2\lambda}.
\end{align}
\end{subequations}
As $Y_j$, $Z_j\in O(\exp(\sqrt{3}\pi k))$, the Riemann-Lebesgue lemma guarantees conditional convergence for all $x\in(0,1)$.

This completes the proof of theorem~\ref{thm:Diag:LKdV:2}.

\begin{rmk} \label{rmk:diag.type2only.LKdV}
We observed above that $0$ is a removable singularity of $F_\lambda$ defined by~\eqref{eqn:introTrans.1.1a},~\eqref{eqn:introTrans.1.1c} and~\eqref{eqn:introTrans.1.1d}. The same holds for $F_\lambda$ defined by~\eqref{eqn:introTrans.1.1a},~\eqref{eqn:introTrans.1.1e} and~\eqref{eqn:introTrans.1.1f}. Hence, for both problems~1 and~2,
\BE
\int_{\Gamma_0}e^{i\lambda x}F_\lambda(f)\d\lambda = 0
\EE
and we could redefine the inverse transform~\eqref{eqn:introTrans.1.1b} as
\begin{align}
F(\lambda) &\mapsto f(x): & f_x(F) &= \left\{ \int_{\Gamma^+} + \int_{\Gamma^-} \right\} e^{i\lambda x} F(\lambda) \d\lambda, \qquad x\in[0,1].
\end{align}
This permits spectral representations of both $S^{(\mathrm{I})}$ and $S^{(\mathrm{II})}$ via augmented eigenfunctions of type~II \emph{only}, that is spectral representations in the sense of definition~\ref{defn:Spect.Rep.I.II} but with $E^{(\mathrm{I})}=\emptyset$.
\end{rmk}

\subsection{General} \label{ssec:Spectral:General}

We will show that the transform pair $(F_\lambda,f_x)$  defined by equations~\eqref{eqn:defn.forward.transform} represents spectral decomposition into type~I and type~II augmented eigenfunctions.

\begin{thm} \label{thm:Diag.2}
Let $S$ be the spatial differential operator associated with a type~\textnormal{II} IBVP. Then the transform pair $(F_\lambda,f_x)$ provides a spectral representation of $S$ in the sense of definition~\ref{defn:Spect.Rep.I.II}.
\end{thm}

The principal tools for constructing families of augmented eigenfunctions are lemma~\ref{lem:GEInt1}, as well as the following lemma:

\begin{lem} \label{lem:GEInt2}
Let $F^\pm_\lambda$ be the functionals defined in equations~\eqref{eqn:defn.Fpm.rho}.
\begin{enumerate}
  \item[(i)]{Let $\gamma$ be any simple closed contour. Then $\{F^\pm_\lambda:\lambda\in\gamma\}$ are families of type~\textup{I} augmented eigenfunctions of $S$ up to integration along $\gamma$ with eigenvalues $\lambda^n$.}
  \item[(ii)]{Let $\gamma$ be any simple closed contour which neither passes through nor encloses $0$. Then $\{F^\pm_\lambda:\lambda\in\gamma\}$ are families of type~\textup{II} augmented eigenfunctions of $S$ up to integration along $\gamma$ with eigenvalues $\lambda^n$.}
  \item[(iii)]{Let $0\leq\theta<\theta'\leq\pi$ and define $\gamma^+$ to be the boundary of the open set
  \BE
  \{\lambda\in\C:|\lambda|>\epsilon, \, \theta<\arg\lambda<\theta'\};
  \EE
  similarly, $\gamma^-$ is the boundary of the open set
  \BE
  \{\lambda\in\C:|\lambda|>\epsilon, \, -\theta'<\arg\lambda<-\theta\}.
  \EE
  Both $\gamma^+$ and $\gamma^-$ have positive orientation. Then $\{F^\pm_\lambda:\lambda\in\gamma^\pm\}$ are families of type~\textup{II} augmented eigenfunctions of $S$ up to integration along $\gamma^\pm$ with eigenvalues $\lambda^n$.}
\end{enumerate}
\end{lem}

\begin{proof} ~ 
\begin{enumerate}
  \item[(i)]{\& (ii) By lemma~\ref{lem:GEInt1}, the remainder functionals are analytic in $\lambda$ within the region bounded by $\gamma$. Cauchy's theorem yields the result.}
  \item[(iii)]{Let $G$ be the union of the set lying to the left of $\gamma^+$ with the set of points lying on $\gamma^+$. Then $G$ lies in the closed upper half-plane. By lemma~\ref{lem:GEInt1},
  \BE
  \int_{\gamma^+} e^{i\lambda x}\lambda^{-n}(F_\lambda^+(Sf) - \lambda^nF_\lambda^+(f))\d\lambda = \int_{\gamma^+} e^{i\lambda x}\lambda^{-n}P^+_f(\lambda)\d\lambda,
  \EE
  and the integrand is the product of $e^{i\lambda x}$ with a function analytic on $G$ and decaying as $\lambda\to\infty$ from within $G$. Hence, by Jordan's lemma, the integral of the remainder functionals vanishes for all $x>0$. For $\gamma^-$, the proof is similar.}
\end{enumerate}
\end{proof}

The contours arising for the operators we consider are each of one of the forms described in lemma~\ref{lem:GEInt2}. Moreover, the augmented eigenfunctions with $\lambda$ on these contours are of a specific type, depending upon the behaviour of the contour. Indeed, the circular contours correspond to augmented eigenfunctions of type~I whereas the infinite contours correspond to augmented eigenfunctions of type~II.

If a circular contour is away from $0$ then the augmented eigenfunctions corresponding to that contour are also of type~II, as stated in lemma~\ref{lem:GEInt2}(ii). There is an analogous result for the infinite contours: if a contour $\gamma^+$, as described in lemma~\ref{lem:GEInt2}(iii), has neither infinite component lying on $\R$, that is $0<\theta<\theta'<\pi$, then the corresponding augmented eigenfunctions are also of type~I. However, for any operator $S$, there is always at least one contour with an infinite component lying on $\R$. For type~II IBVP, it is impossible to deform this contour away from $\R$ and the augmented eigenfunctions would not form a complete system without including those with $\lambda$ on this contour. If $\gamma$ has an infinite component along $\R$ and $R_\lambda$ is the polynomial $P_f^\pm$ remainder functional described in lemma~\ref{lem:GEInt1}, then the integral~\eqref{eqn:defnAugEig.Control1} will diverge. However, the division by $\lambda^n$ is sufficient to obtain convergence of the integral~\eqref{eqn:defnAugEig.Control2}.

\medskip

Let $(S,a)$ be such that the associated IBVP is well-posed. Then there exists a complete system of augmented eigenfunctions associated with $S$, some of which are type~\textup{I} whereas the rest are type~\textup{II}. Indeed:

\begin{prop} \label{prop:GEIntComplete3}
The system
\BE
\mathcal{F}_0 = \{F^+_\lambda:\lambda\in\Gamma^+_0\}\cup\{F^-_\lambda:\lambda\in\Gamma^-_0\}
\EE
is a family of type~\textup{I}  augmented eigenfunctions of $S$ up to integration over $\Gamma_0$, with eigenvalues $\lambda^n$. The system
\BE
\mathcal{F}_a = \{F^+_\lambda:\lambda\in\Gamma^+_a\}\cup\{F^-_\lambda:\lambda\in\Gamma^-_a\}
\EE
is a family of type~\textup{II} augmented eigenfunctions of $S$ up to integration over $\Gamma_a$, with eigenvalues $\lambda^n$.

Furthermore, if an IBVP associated with $S$ is well-posed, then $\mathcal{F}=\mathcal{F}_0\cup \mathcal{F}_a$ is a complete system.
\end{prop}

\begin{proof}
Considering $f\in\Phi$ as the initial datum of the homogeneous IBVP and applying proposition~\ref{prop:Transform.Method:General:2}, we evaluate the solution of problem~\eqref{eqn:IBVP} at $t=0$,
\BE
f(x) = q(x,0) = \int_{\Gamma_0^+} e^{i\lambda x} F^+_\lambda(f) \d\lambda + \int_{\Gamma_0^-} e^{i\lambda x} F^-_\lambda(f) \d\lambda.
\EE
Thus, if $F^\pm_\lambda(f)=0$ for all $\lambda\in\Gamma_0$ then $f=0$.

By lemma~\ref{lem:GEInt2}~(i), $\mathcal{F}_0$ is a system of type~I augmented eigenfunctions up to integration along $\Gamma^+_0\cup\Gamma^-_0$.

Applying lemma~\ref{lem:GEInt1} to $\mathcal{F}_a$, we obtain
\BE
F^\pm_\lambda(Sf) = \lambda^n F^\pm_\lambda(f) + R^\pm_\lambda(f),
\EE
with
\BE
R^+_\lambda(f) = P^+_f(\lambda), \qquad R^-_\lambda(f) = P^-_f(\lambda)e^{-i\lambda}.
\EE
By lemma~\ref{lem:GEInt2}~(ii), we can deform the contours $\Gamma^\pm_a$ onto the union of several contours of the form of the $\gamma^\pm$ appearing in lemma~\ref{lem:GEInt2}~(iii). The latter result completes the proof.
\end{proof}

\begin{proof}[Proof of theorem~\ref{thm:Diag.2}]
Proposition~\ref{prop:GEIntComplete3} establishes completeness of the augmented eigenfunctions and equations~\eqref{eqn:Spect.Rep.II}, under the assumption that the integrals converge. The series of residues
\BE
\int_{\Gamma_0} e^{i\lambda x} F^\pm_\lambda(Sf) \d\lambda = 2\pi i\sum_{\substack{\sigma\in\C:\\\Delta(\sigma)=0}}e^{i\sigma x}\res_{\lambda=\sigma}F^\pm_\lambda(Sf),
\EE
whose convergence is guaranteed by the well-posedness of the IBVP~\cite{Smi2013a}. Indeed, a necessary condition for well-posedness is the convergence of this series for $Sf\in\Phi$. But then the definition of $F^\pm_\lambda$ implies
\BES
\res_{\lambda=\sigma}F^\pm_\lambda(f) = O(|\sigma|^{-j-1}), \mbox{ where } j=\max\{k:\hsforall f\in\Phi,\; f^{(k)}(0)=f^{(k)}(1)=0\},
\EES
so $\res_{\lambda=\sigma} F_\lambda(Sf) = O(|\sigma|^{-1})$ and the Riemann-Lebesgue lemma gives convergence. This verifies statement~\eqref{eqn:Spect.Rep.defnI.II.conv1}. Theorem~\ref{prop:Transforms.valid:General:2} ensures convergence of the right hand side of equation~\eqref{eqn:Spect.Rep.II.2}. Hence statement~\eqref{eqn:Spect.Rep.defnI.II.conv2} holds.
\end{proof}

\begin{rmk}
Suppose $S$ is a type~I operator.

By the definition of a type~I operator (more precisely, by the properties of an associated type~I IBVP, see~\cite{Smi2013a}), $F^\pm_\lambda(\phi)=O(\lambda^{-1})$ as $\lambda\to\infty$ within the sectors interior to $\Gamma^\pm_a$. Hence, by Jordan's lemma,
\BE
\int_{\Gamma^+_a}e^{i\lambda x} F^+_\lambda(\phi)\d\lambda + \int_{\Gamma^-_a}e^{i\lambda x} F^-_\lambda(\phi)\d\lambda = 0.
\EE
Hence, it is possible to define an alternative inverse transform
\begin{align} \label{eqn:defn.inverse.transform.typeI}
F(\lambda) &\mapsto f(x): & f^\Sigma_x(F) &= \int_{\Gamma_0} e^{i\lambda x} F(\lambda)\d\lambda,
\end{align}
equivalent to $f_x$. The new transform pair $(F_\lambda,f^\Sigma_x)$ defined by equations~\eqref{eqn:defn.forward.transform} and~\eqref{eqn:defn.inverse.transform.typeI} may be used to solve an IBVP associated with $S$ so 
\BE
\mathcal{F}_0 = \{F^+_\lambda:\lambda\in\Gamma^+_0\}\cup\{F^-_\lambda:\lambda\in\Gamma^-_0\}
\EE
is a complete system of functionals on $\Phi$.

Moreover, $\mathcal{F}_0$ is a family of type~I augmented eigenfunctions \emph{only}. Hence, $\mathcal{F}_0$ provides a spectral representation of $S$ in the sense of definition~\ref{defn:Spect.Rep.I}. Via a residue calculation at each zero of $\Delta$, one obtains a classical spectral representation of $S$ as a series of (generalised) eigenfunctions.

We emphasize that this spectral representation without type~II augmented eigenfunctions is only possible for a type~I operator.
\end{rmk}

\begin{rmk}
By definition, the point $3\epsilon/2$ is always exterior to the set enclosed by $\Gamma$. Therefore introducing a pole at $3\epsilon/2$ does not affect the convergence of the contour integral along $\Gamma$. This means that the system $\mathcal{F}'=\{(\lambda-3\epsilon/2)^{-n}F_\lambda:\lambda\in\Gamma\}$ is a family of type~\textup{I} augmented eigenfunctions, thus no type~\textup{II} augmented eigenfunctions are required; equation~\eqref{eqn:Spect.Rep.I} holds for $\mathcal{F}'$ and the integrals converge. However, we cannot show that $\mathcal{F}'$ is complete, so we do not have a spectral representation of $S$ through the system $\mathcal{F}'$.
\end{rmk}

\begin{rmk} \label{rmk:Splitting.of.rep}
There may be infinitely many circular components of $\Gamma_a$, each corresponding to a zero of $\Delta$ which lies in the interior of a sector enclosed by the main component of $\Gamma_a$. It is clear that in equations~\eqref{eqn:Transforms.valid:General:prop2} and~\eqref{eqn:Transform.Method:General:prop2:1}, representing the validity of the transform pair and the solution of the IBVP, the contributions of the integrals around these circular contours are cancelled by the contributions of the integrals around certain components of $\Gamma_0$, as shown in figure~\ref{fig:general-contdef}. Hence, we could redefine the contours $\Gamma_a$ and $\Gamma_0$ to exclude these circular components without affecting the validity of propositions~\ref{prop:Transforms.valid:General:2} and~\ref{prop:Transform.Method:General:2}.

Our choice of $\Gamma_a$ is intended to reinforce the notion that $S$ is split into two parts by the augmented eigenfunctions. In $\Gamma_0$, we have chosen a contour which encloses each zero of the characterstic determinant individually, since each of these zeros is a classical eigenvalue, so $\mathcal{F}_0$ corresponds to the set of all generalised eigenfunctions. Hence $\mathcal{F}_a$ corresponds only to the additional spectral objects necessary to form a complete system.
\end{rmk}

\begin{rmk} \label{rmk:Gamma.a.ef.at.inf}
As $\Gamma_a$ encloses no zeros of $\Delta$, we could choose any $R>0$ and redefine $\Gamma^\pm_{a\Mspacer R}$ as the boundary of
\BE
\left\{\lambda\in\C^\pm:|\lambda|>R,\;\Re(a\lambda^n)>0\right\} \setminus \bigcup_{\substack{\sigma\in\C:\\\Delta(\sigma)=0}}D(\sigma,2\epsilon),
\EE
deforming $\Gamma_a$ over a finite region. By considering the limit $R\to\infty$, we claim that $\mathcal{F}_a$ can be seen to represent \emph{spectral objects with eigenvalue at infinity}.
\end{rmk}

\begin{rmk} \label{rmk:diag.type2only.General}
By lemma~\ref{lem:GEInt2}(ii), for all $\sigma\neq0$ such that $\Delta(\sigma)=0$, it holds that $\{F^\pm_\lambda:\lambda\in C(\sigma,\epsilon)\}$ are families of type~II augmented eigenfunctions. Hence, the only component of $\Gamma_0$ that may not be a family of type~II augmented eigenfunctions is $C(0,\epsilon)$. If
\begin{subequations}
\begin{align}
\gamma_a^+ &= \Gamma_a^+ \setminus \displaystyle\bigcup_{\substack{\sigma\in\overline{\C^+}:\\\sigma\neq0,\\\Delta(\sigma)=0}}C(\sigma,\epsilon), \\
\gamma_a^- &= \Gamma_a^- \setminus \displaystyle\bigcup_{\substack{\sigma\in\C^-:\\\Delta(\sigma)=0}}C(\sigma,\epsilon), \\
\gamma_0 &= C(0,\epsilon),
\end{align}
\end{subequations}
then
\BE
\mathcal{F}'_a = \{F^+_\lambda:\lambda\in\gamma_a^+\}\cup\{F^-_\lambda:\lambda\in\gamma_a^-\}
\EE
is a family of type~II augmented eigenfunctions and
\BE
\mathcal{F}'_0 = \{F^+_\lambda:\lambda\in\gamma_0\}
\EE
is a family of type~I augmented eigenfunctions of $S$. For $S$ type~I or type~II, $\mathcal{F}'_a\cup\mathcal{F}'_0$ provides a spectral representation of $S$ in the sense of definition~\ref{defn:Spect.Rep.I.II}, with minimal type~I augmented eigenfunctions.

Assume that $0$ is a removable singularity of $F^+_\lambda$. Then $\mathcal{F}'_a$ provides a spectral representation of $S$ in the sense of definition~\ref{defn:Spect.Rep.I.II} with $E^{(\mathrm{I})}=\emptyset$. We have already identified the operators $S^{(\mathrm{I})}$ and $S^{(\mathrm{II})}$ for which this representation is possible (see Remark~\ref{rmk:diag.type2only.LKdV}).
\end{rmk}

\begin{rmk} \label{rmk:Ill-posed}
The validity of lemmata~\ref{lem:GEInt1} and~\ref{lem:GEInt2} does not depend upon the class to which $S$ belongs. Hence, even if all IBVP associated with $S$ are ill-posed, it is still possible to construct families of augmented eigenfunctions of $S$. However, without the well-posedness of an associated IBVP, an alternative method is required in order to analyse the completeness of these families. Without completeness results, it is impossible to discuss diagonalisation of the operator by augmented eigenfunctions.
\end{rmk}

\begin{rmk} \label{rmk:Lower-Order-Terms}
It is expected that the above results could be carried over to IBVP for PDE with lower order terms, at least where the corresponding spatial differential operator remains formally self-adjoint. In this case, the eigenvalue is not simply $\lambda^n$ but $z(\lambda)$, for $z$ the symbol of the spatial differential operator.
\end{rmk}

\section{Conclusion}

In the classical separation of variables, 
one makes a particular assumption on the form of the solution. For evolution PDEs in one dimension, this is usually expressed as
\begin{quote}
``Assume the solution takes the form $q(x,t)=\tau(t)\xi(x)$ for all $(x,t)\in[0,1]\times[0,T]$ for some $\xi\in C^\infty[0,1]$ and $\tau\in C^\infty[0,T]$.''
\end{quote}
However, when applying the boundary conditions, one superimposes infinitely many such solutions. So it would be more accurate to use the assumption
\begin{quote}
``Assume the solution takes the form $q(x,t)=\sum_{m\in\N}\tau_m(t)\xi_m(x)$ for some sequences of functions $\xi_m\in C^\infty[0,1]$ which are eigenfunctions of the spatial differential operator, and $\tau_m\in C^\infty[0,T]$; assume that the series converges uniformly for $(x,t)\in[0,1]\times[0,T]$.''
\end{quote}
For this `separation of variables' scheme to yield a result, we require completeness of the eigenfunctions $(\xi_m)_{m\in\N}$ in the space of admissible initial data, and convergence of the series.

The concept of generalized eigenfunctions, as presented by Gelfand and coauthors~\cite{GS1967a,GV1964a} allows one to weaken the above assumption in two ways: first, it allows the index set to be uncountable, hence the series is replaced by an integral. Second, certain additional spectral functions, which are not genuine eigenfunctions, are admitted to be part of the series.

An integral expansion in generalized eigenfunctions is insufficient to describe the solutions of IBVP obtained via the unified transform method for type~II problems. In order to describe these IBVP, we have introduced type~II augmented eigenfunctions. Using these new eigenfunctions, the assumption is weakened further:
\begin{quote}
``Assume the solution takes the form $q(x,t)=\int_{m\in\Gamma}\tau_m(t)\xi_m(x)\d m$ for some functions $\xi_m\in C^\infty[0,1]$, which are type~I and~II augmented eigenfunctions of the spatial differential operator, and $\tau_m\in C^\infty[0,T]$; assume that the integral converges uniformly for $(x,t)\in[0,1]\times[0,T]$.''
\end{quote}

It appears that it is \emph{not} possible to weaken the above assumption any further. Indeed, it has been established in~\cite{FP2001a} that the unified method provides the solution of \emph{all} well-posed problems. Combining propositions~\ref{prop:Transform.Method:General:2} and~\ref{prop:GEIntComplete3}, we can replace the above assumption with the following theorem:
\begin{quote}
``Suppose $q(x,t)$ is the $C^\infty$ solution of a well-posed two-point linear constant-coefficient IBVP. Then $q(x,t)=\int_{m\in\Gamma}\tau_m(t)\xi_m(x)\d m$, where $\xi_m\in C^\infty[0,1]$ are type~I and~II augmented eigenfunctions of the spatial differential operator and $\tau_m\in C^\infty[0,T]$ are some coefficient functions. The integral converges uniformly for $(x,t)\in[0,1]\times[0,T]$.''
\end{quote}

In summary, both type~I and type~II IBVP admit integral representations like~\eqref{eqn:introIBVP.solution.2}, which give rise to transform pairs associated with a combination of type~I and type~II augmented eigenfunctions. For type~I IBVP, it is possible (by appropriate contour deformations) to obtain alternative integral representations like~\eqref{eqn:introIBVP.solution.1}, which give rise to transform pairs associated with only type~I augmented eigenfunctions. Furthermore, in this case, a residue calculation yields a classical series representation, which can be associated with Gel'fand's generalised eigenfunctions.

\subsubsection*{Acknowledgement}
The research leading to these results has received funding from the European Union's Seventh Framework Programme FP7-REGPOT-2009-1 under grant agreement n$^\circ$ 245749.

\bibliographystyle{amsplain}
\bibliography{dbrefs}

\end{document}